\documentclass[12pt, reqno]{amsart}
\usepackage{amsmath}
\usepackage{amssymb}
\usepackage{amsthm}
\usepackage{xcolor}
\usepackage[abbrev]{amsrefs}
\usepackage[utf8]{inputenc}
\usepackage{enumitem}
\usepackage[normalem]{ulem}

\newcommand{\A}{\ensuremath{\mathcal{A}}}
\newcommand{\VV}{\ensuremath{\mathcal{V}}}
\newcommand{\XX}{\ensuremath{\mathcal{X}}}
\newcommand{\YY}{\ensuremath{\mathcal{Y}}}
\newcommand{\ZZ}{\ensuremath{\mathcal{Z}}}
\newcommand{\UU}{\ensuremath{\mathcal{U}}}
\newcommand{\OO}{\ensuremath{\mathcal{O}}}
\newcommand{\II}{\ensuremath{\mathcal{I}}}
\newcommand{\Lip}{\ensuremath{\mathrm{Lip}}}
\newcommand{\Rea}{\ensuremath{\mathbb{R}}}
\newcommand{\Nat}{\ensuremath{\mathbb{N}}}
\newcommand{\Int}{\ensuremath{\mathbb{Z}}}

\newcommand{\MM}{\ensuremath{\mathcal{M}}}
\newcommand{\NN}{\ensuremath{\mathcal{N}}}
\newcommand{\F}{\ensuremath{\mathcal{F}}}
\newcommand{\Id}{\operatorname{\rm{Id}}}
\newcommand{\supp}{\operatorname{\mathrm{supp}}\nolimits}

\newcommand{\spn}{\operatorname{\mathrm{span}}\nolimits}

\newtheorem{Theorem}{Theorem}[section]
\newtheorem{Lemma}[Theorem]{Lemma}
\newtheorem{Proposition}[Theorem]{Proposition}
\newtheorem{Corollary}[Theorem]{Corollary}

\theoremstyle{remark}
\newtheorem{Remark}[Theorem]{Remark}
\newtheorem{Definition}[Theorem]{Definition}

\newtheorem{Question}[Theorem]{Question}

\allowdisplaybreaks

\AtBeginDocument{\def\MR#1{}}



\begin{document}

\title[Sums of Lipschitz free spaces]{Lipschitz free spaces isomorphic to their infinite sums and geometric applications}

\author[F. Albiac]{Fernando Albiac}
\address{Mathematics Department--InaMat \\
Universidad P\'ublica de Navarra\\
Campus de Arrosad\'{i}a\\
Pamplona\\
31006 Spain}
\email{fernando.albiac@unavarra.es}

\author[J. L. Ansorena]{Jos\'e L. Ansorena}
\address{Department of Mathematics and Computer Sciences\\
Universidad de La Rioja\\
Logro\~no\\
26004 Spain}
\email{joseluis.ansorena@unirioja.es}

\author[M. C\'uth]{Marek C\'uth}
\address{Faculty of Mathematics and Physics, Department of Mathematical Analysis\\
Charles University\\
186 75 Praha 8\\
Czech Republic}
\email{cuth@karlin.mff.cuni.cz}

\author[M. Doucha]{Michal Doucha}
\address{Institute of Mathematics\\
Czech Academy of Sciences\\
\v Zitn\'a 25\\
115 67 Praha 1\\
Czech Republic}
\email{doucha@math.cas.cz}

\begin{abstract}We find general conditions under which Lipschitz-free spaces over metric spaces are isomorphic to their infinite direct $\ell_1$-sum and exhibit several applications. As examples of such applications we have that Lipschitz-free spaces over balls and spheres of the same finite dimensions are isomorphic, that the Lipschitz-free space over $\Int^d$ is isomorphic to its $\ell_1$-sum, or that the Lipschitz-free space over any snowflake of a doubling metric space is isomorphic to $\ell_1$. Moreover, following new ideas of Bru{\`e} et al.\ from \cite{BMS18} we provide an elementary self-contained proof that Lipschitz-free spaces over doubling metric spaces are complemented in Lipschitz-free spaces over their superspaces and they have BAP. Everything, including the results about doubling metric spaces, is explored in the more comprehensive setting of $p$-Banach spaces, which allows us to appreciate the similarities and differences of the theory between the cases $p<1$ and $p=1$.
\end{abstract}

\thanks{F. Albiac acknowledges the support of the Spanish Ministry for Economy and Competitivity under Grant MTM2016-76808-P for \emph{Operators, lattices, and structure of Banach spaces}. F. Albiac and J.~L. Ansorena acknowledge the support of the Spanish Ministry for Science, Innovation, and Universities under Grant PGC2018-095366-B-I00 for \emph{An\'alisis Vectorial, Multilineal y Aproximaci\'on}. M.~C\'uth has been supported by Charles University Research program No. UNCE/SCI/023. M. Doucha was supported by the GA\v{C}R project 19-05271Y and RVO: 67985840.}

\maketitle

\section{Introduction}\noindent
In recent years, Lipschitz-free spaces over metric spaces have become one of the most widely investigated class of Banach spaces. They are intimately connected to the nonlinear geometry of Banach spaces and have proved a very useful tool in their study, to the extend that they have become a topic towards which the modern research focus in Banach space theory is shifting; see, e.g., the seminal paper \cite{GodefroyKalton2003}, the surveys \cites{GLZ, G15}, or the monograph \cite{WeaverBook2018}*{Chapter 5}.

The universal property of the norm of these spaces relates them to similar objects from different areas of mathematics such as, notably, optimal transport, where the set of Radon probability measures with finite first momentum endowed with the \emph{Wasserstein distance}, known as the \emph{Wasserstein space} (see the monograph by Villani \cite{VillaniBook2}), is canonically isometric to a set whose linear hull is dense in the corresponding Lipchitz-free space (see \cite{Edwards2011}*{Theorem 4.1, Theorem 4.4, Theorem 4.5 and Theorem 6.1}). Another research subject closely connected to the study of approximation properties of Lipschitz-free spaces which recently attracted a considerable attention of many researchers (see \cite{BrudnyiBrudnyi2007} or \cite{LN05}) is the topic of finding linear extension operators between Banach spaces of Lipschitz functions (see, e.g., \cite{K15}*{Section 2.1}).

Even within Banach space theory, these spaces are known under several different names, like Arens-Eells spaces (\cite{WeaverBook2018}) or transportation cost spaces (\cite{OO19}) to name a few. Here we will stick to the term Lipschitz-free space, which was the one used by Godefroy and Kalton in \cite{GodefroyKalton2003} and which is now prevalent in the theory. For more historical and terminological remarks we refer the reader to \cite{OO19}*{Section 1.6}.

There is an analogue of Lipschitz-free spaces in the more general framework of $p$-Banach spaces for $p\in(0,1]$, namely, the \emph{Lipschitz free $p$-spaces over quasimetric spaces}, where the case $p=1$ corresponds to the locally convex members of this extended family, i.e., the standard Lipschitz free Banach spaces over metric spaces. Lipschitz free $p$-spaces over quasimetric spaces were introduced in \cite{AlbiacKalton2009} with the purpose to build examples for each $0<p<1$ of two $p$-Banach spaces which are Lipschitz isomorphic but not linearly isomorphic. More recently, the authors initiated a systematic study of the structure of this class of spaces in \cites{AACD2018, AACD2019}.

A specific feature of nonlocally convex $p$-Banach spaces that differentiates them from Banach spaces is the lack of a duality theory, which forces the study of Lipschitz free $p$-Banach spaces over quasimetric (or metric) spaces to rely on more geometrical methods. It turns out that this approach brings fresh wind even to better understand the case $p=1$, as it is supported by several new results obtained in \cites{AACD2018, AACD2019}.

In this paper we continue with this line of research. Our motivation is to generalize several important structural results that have become a standard toolkit of a Lipschitz-free space theorist. These include Lee and Naor's result \cite{LN05} on the existence of $K$-random partitions of unity with respect to any subspace of a doubling metric space, Kalton's result from \cite{Kalton2004} that every Lipschitz-free space embeds into the infinite direct $\ell_1$-sum of free spaces over its annuli; Kaufmann's results from \cite{K15} that for every Banach space $X$ the Lipschitz-free space over $X$ is, on one hand, isomorphic to its $\ell_1$-sum and, on the other hand, to the Lipschitz-free space over its unit ball. We also extend the results from the paper \cite{CDL19}, which are formulated only in terms of the duals of Lipschitz-free spaces. Not only did we succeed to prove analogues of all these results for Lipschitz free $p$-spaces with $p\in(0,1]$ but, revisiting the topic, we actually found new results and applications which are of interest even for the classical case $p=1$.

We get started with the notion of \emph{$p$-complementably amenable subspace with constant $K$,} which for $p=1$ is equivalent to the existence of a $K$-random projection, a notion suggested by Ambrosio and Puglisi from \cite{AmbrosioPuglisi2016} and which in turn is motivated by the notion of $K$-gentle partition of unity by Lee and Naor from \cite{LN05}. Our main general results, which are novel even for the case $p=1$, are developed in Section~\ref{sect:generalEllPSum}. The applications of our methods, some of which are also new even for $p=1$, are explained in Section~\ref{sec:appl}. Let us advance some of the most interesting practical implementations of our results to the general theory:
\begin{enumerate}[label={$\bullet$},leftmargin=*]
\item If $\MM$ is a metric space and there is a point $x\in\MM$ such that every annulus centered at $x$ contains only finitely many points, then for every $p\in(0,1]$ the space
$\F_p(\MM)$ admits a subsymmetric basis if and only if $\F_p(\MM)$ is isomorphic to $\ell_p$ (see Proposition~\ref{prop:FpSubsym}).

\item Generalizing \cite{K15}*{Theorem 3.1 and Corollary 3.3}, if $X$ is a Banach space and $\MM\subset X$ is closed under multiplication by nonnegative numbers, for every $p\in(0,1]$ we have
\[
\F_p(\MM)\simeq \F_p(B_\MM) \simeq \F_p(\MM\setminus B_\MM) \simeq {\ell_p}(\F_p(\MM))
\]
(see Theorem~\ref{thm:Banach}).
\item For $d\in\Nat$ and $p\in(0,1]$ we have\[\F_p(\Rea^d)\simeq \F_p(S_{\Rea^{d+1}})\]
(see Theorem~\ref{thm:sphere}).
\end{enumerate}
In Section~\ref{sec:doubling} we show that whenever $\NN$ is a doubling metric space and $\MM$ is a metric space containing $\NN$ then $\F_p(\NN)$ is complemented in $\F_p(\MM)$ for every $p\in(0,1]$, which in particular solves Question 6.7 from \cite{AACD2019}. This is known for the case $p=1$, but we show an argument, based on ideas from \cite{BMS18}, which is much simpler (and self-contained) than the one used by other authors to prove it for $p=1$. We provide several applications to doubling metric spaces, amongst which we highlight the following ones because of their interest. Note that best-known examples of doubling metric spaces are, e.g., subsets of $\Rea^d$ or, more generally, subsets of Carnot groups.
\begin{enumerate}[label={$\bullet$},leftmargin=*]
\item If $(\MM,d)$ is a doubling metric space and $\alpha\in(0,1)$, then
\[
\F(\MM,d^\alpha)\simeq \ell_1
\]
(see Corollary~\ref{cor:doublingSnowflake}). This improves a result previously known only for compact doubling metric spaces (see\cite{WeaverBook2018}*{Theorems 4.38 and 8.49}).
\item If $\MM$ is a doubling self-similar metric space (like for instance a finite-dimensional Banach space or more generally a Carnot group) and $\NN\subset \MM$ is a net in $\MM$, then for every $p\in(0,1]$ we have
\[
\F_p(\NN)\simeq {\ell_p} (\F_p(\NN))
\]
and
\[
\Lip_0(\MM)\simeq\Lip_0(\NN).
\]
Those results are improvements of \cite{HN17}*{Theorem 7} and \cite{CDL19}*{Corollary 1.18}.
\item For every $d\in\Nat$ and $p\in(0,1]$ we have\[
\F_p(\Int^d)\simeq\F_p(\Nat^d)
\](see Theorem~\ref{thm:natAndInt}).
\end{enumerate}

\section{Notation and preliminaries}\noindent
We use standard notation in Banach space theory as can be found in \cite{AlbiacKalton2016}. We refer the reader to \cite{AACD2018}*{Sections 2 and 4} for basic facts and notation concerning $p$-metric spaces, $p$-Banach spaces, and Lipschitz free $p$-spaces over them for $0<p\le 1$.

We put
\[
\Nat_*=\{n\in\Int \colon n\ge 0\}.
\]

If $(\MM,d,0)$ is an arbitrary pointed $p$-metric space and $A\subseteq (0,\infty)$ (usually $A$ is an interval) we put
\[
\MM_A=\{x\in\MM \colon d(0,x)\in A\}\]
and
\[\MM_A^*=\{0\}\cup \MM_A.
\]
Given $A\subset\Rea$ and $R>0$ we put
\[
R^{A}=\{R^x\colon x\in A\}.
\]

Given a quasi-Banach space $X$ and $A\subset X$, we denote by $[A]$ the closed linear span of $A$. Moreover, for $m\in\Nat$ we denote by $\kappa_m(X)$ the smallest constant $C\ge 1$ such that
\[
\left\Vert \sum_{j=1}^m x_j \right\Vert \le C \sum_{j=1}^m \Vert x_j\Vert,
\quad x_j\in X.
\]
Note that if $X$ is a $p$-Banach space then $\kappa_m(X)\le m^{1/p-1}$.

Given $0\le p\le\infty$ and a countable set $N$, $(\bigoplus_{n\in N} X_n)_p$ denotes the sum in the sense of $\ell_p$ ($c_0$ if $p=0$) of the family of quasi-Banach spaces $(X_n)_{n\in N}$. If all the spaces $X_n$ are equal to a single space $X$, we shall instead denote their $p$-sum by $\ell_p(X)$. If $N=\{1,2\}$ we will simply put $(X_1\oplus X_2)_p$, and if the index $p$ is irrelevant or clear from context $X_1\oplus X_2$ will stand for $(X_1\oplus X_2)_p$.

\begin{Definition} Let $0<p\le 1$.
We say that two $p$-metric spaces $\NN$ and $\MM$ are $K$-Lipschitz isomorphic, where $K\in[1,\infty)$, if there are Lipschitz maps $f\colon \NN\to \MM$ and $g\colon \MM\to\NN$ such that $g\circ f=\Id_\NN$, $f\circ g=\Id_\MM$ and $\Vert f \Vert_\Lip \Vert g\Vert_\Lip\le K$. If the constant $K$ is irrelevant, we say that $\NN$ and $\MM$ are Lipschitz isomorphic.
\end{Definition}

\begin{Definition}\label{def:isomorphism} We say that two quasi-Banach spaces $X$ and $Y$ are $K$-isomorphic and write $X\simeq_K Y$, where $K\in[1,\infty)$, if there are bounded linear maps $S\colon X \to Y$ and $T\colon Y \to X$ such that $T\circ S=\Id_X$, $T\circ S=\Id_Y$, and $\Vert T\Vert\cdot \Vert S\Vert\le K$. If the constant $K$ is irrelevant, we say that $X$ and $Y$ are isomorphic and write $X\simeq Y$.
\end{Definition}

\begin{Remark}
If two quasi-Banach spaces $X$ and $Y$ are $K$-isomorphic, we can choose $S$ as in Definition~\ref{def:isomorphism} with $\Vert S\Vert=1$. Thus, if $X$ and $Y$ are $1$-isomorphic, in particular they are isometric. In contrast, two $1$-Lipschitz isomorphic $p$-metric spaces need not be isometric, as shown, e.g., by the metric spaces $[0,1]$ and $[0,2]$ with the usual distance.
\end{Remark}

\begin{Definition}
Given quasi-Banach spaces $X$ and $Y$, a constant $K\geq 1$, and a linear map $S\colon X\to Y$, we say that \emph{$X$ is $K$-complemented in $Y$ via $S$} if there exists a bounded linear operator $T\colon Y\to X$ with $T\circ S = \Id_X$ and $\|S\|\cdot \|T\|\leq K$. We will say that \emph{$X$ is $K$-complemented in $Y$}, and write $X\unlhd_K Y$, if there exists a linear map $S$ such that $X$ is $K$-complemented in $Y$ via $S$. If the constant $K$ is irrelevant, we will say that $X$ is complemented in $Y$ and write $X\unlhd Y$. If $X$ is isomorphic to a (not necessarily complemented) subspace of $Y$ we will put $X\lesssim Y$.
\end{Definition}
Of course, a quasi-Banach space $X$ is complemented in a quasi-Banach space $Y$ if and only if there is a quasi-Banach space $X_0$ such that $Y\simeq X\oplus X_0$.

\subsection{Complementable $p$-amenability.} Suppose $\MM$ and $\NN$ are $p$-metric spaces, $0<p\le 1$. If
$f\colon\MM\to\NN$ is Lipschitz, there is unique bounded linear map $L_f\colon\F_p(\MM)\to\F_p(\NN)$ such that $L_f(\delta_\MM(x)=\delta_\NN(f(x))$ for all $x\in\MM$. The operator $L_f$ is called the canonical linearization of $f$. If $\NN$ is a subset of $\MM$ and $p=1$, the canonical linearization $L_\jmath\colon \F(\NN) \to \F(\MM)$ of the incluison map $\jmath\colon\NN\to\MM$ is an isometric embedding. As this property does not carry out to Lipschitz free $p$-spaces for $p<1$ (see \cite{AACD2018}*{Theorem 6.1 and Question 6.2}), we introduce the following definition.

\begin{Definition} We say that a subset $\NN$ of a $p$-metric space $\MM$ is \emph{$p$-amenable} in $\MM$ with constant $C<\infty$ if $L_{\jmath}$ is an isomorphism and $\Vert L_\jmath^{-1}\Vert\le C$, where $L_\jmath$ is the canonical linearization of the inclusion map $\jmath\colon\NN\to\MM$.
\end{Definition}

Note that $L_\jmath(\F_p(\NN))$ need not be a complemented subspace of $\F(\MM)$ even in the case $p=1$, when every subset $\NN$ of a metric space $\MM$ is $p$-amenable with constant $1$. Thus, we introduce the complemented version of the notion of amenability.

\begin{Definition} We say that a subset $\NN$ of a $p$-metric space $\MM$ is \emph{complementably $p$-amenable} in $\MM$ with constant $C$ if $\F_p(\NN)$ is $C$-complemented in $\F_p(\MM)$ via the canonical linear map from $\F_p(\NN)$ into $\F_p(\MM)$.
\end{Definition}
If $\NN$ is complementably $p$-amenable in $\MM$ with constant $C$, then $\NN$ is $p$-amenable in $\MM$ with constant $C$. If there is a Lipschitz map $r\colon\MM\to\NN$ with $r\circ\jmath=\Id_\NN$, i.e., $r$ is a Lipschitz retraction of $\MM$ onto $\NN$, then $\NN$ is complementably $p$-amenable in $\MM$ with constant $\Vert r\Vert_\Lip$ for every $0<p\le 1$ (see \cite{AACD2018}*{Lemma 4.19}).

If $\NN$ is (complementably) $p$-amenable in $\MM$ with constant $C$, then $\NN$ is (complementably) $p$-amenable in $\MM'$ with constant $C$ for every $\NN\subset\MM'\subset\MM$. Finally, note that if $\NN$ is (complementably) $p$-amenable in $\MM$ with constant $C$, then the same holds for $\overline{\NN}$, which easily implies that for any $\NN'\subset \MM$ with $\overline{\NN'} = \overline{\NN}$, $\NN'$ is (complementably) $p$-amenable in $\MM$ with constant $C$ as well.

The following elementary lemma characterizes complementable $p$-amenability and will be frequently used.
\begin{Lemma}\label{lem:complementedCondition}
Let $(\MM,d,0)$ be a pointed $p$-metric space. Suppose that $\NN$ is a subset of $\MM$ with $0\in\NN$. Then $\NN$ is complementably $p$-amenable in $\MM$ with constant $C>0$ if and only if there exists a $C$-Lipschitz function $f\colon\MM\to \F_p(\NN)$ such that $f(x)=\delta_\NN(x)$ for every $x\in \NN$.
\end{Lemma}

\begin{proof}One implication is trivial, so let us assume the existence of a $C$-Lipschitz function $f\colon\MM\to \F_p(\NN)$ with $f(x)=\delta_\NN(x)$ for all $x\in \NN$. By the universal property of $\F_p(\MM)$, there exists a linear operator $L_{f}\colon\F_p(\MM)\to \F_p(\NN)$ with $L_{f}\circ \delta_\MM = f$ and $\|L_{f}\| = \|f\|_{\Lip}$. We have $L_{f}\circ L_{\jmath}(\delta_\NN(x)) = f(x) = \delta_\NN(x)$ for all $x\in\NN$ and so $L_{f}\circ L_{\jmath} = \Id_{\F_p(\NN)}$.
\end{proof}

Note that if $\NN$ and $\MM$ are metric spaces, then by \cite{AmbrosioPuglisi2016}*{Definition 2.9}, $\NN$ admits a $K$-random projection on $\MM$ if there exists a $K$-Lipschitz function $f\colon\MM\to \F(\NN)$ such that $f(x)=\delta_\NN(x)$ for all $x\in \NN$. Hence, by Lemma~\ref{lem:complementedCondition}, $\NN$ is complementably $1$-amenable in $\MM$ with constant $K$ if and only if $\NN$ admits $K$-random projection on $\MM$.

When we are interested in studying the isomorphic structure of a Lipschitz free $p$-space over a $p$-metric space $\MM$, by the following result we may choose to omit finitely many points from $\MM$ without altering the resulting Lipschitz free $p$-space over $\MM$. Moreover, when discarding one point only, we obtain a uniform estimate independent on the $p$-metric space (the exact value of the constant will not be important).

\begin{Lemma}\label{lem:separatedpoint}Let $(\MM,d)$ be a $p$-metric space, $0<p\leq 1$, and let $x_0\in\MM$. For every $C>3^{1/p}\cdot 5^{2/p}$ we have:
\begin{enumerate}[label={(\roman*)}, leftmargin=*,widest=iii]
\item\label{it:omitPoint} $\MM\setminus\{x_0\}$ is complementably $p$-amenable in $\MM$ with constant $C$.
\item\label{it:omitPoint1dim} If $x_0\in\MM$ is an isolated point, or $\MM$ is a metric space, then $\F_p(\MM)\simeq_{C} \Rea\oplus\F_p(\MM\setminus\{x_0\})$.
\item\label{it:omitPointIso} If $\MM$ is infinite metric space, then $\F_p(\MM)\simeq_{C} \F_p(\MM\setminus\{x_0\})$.
\item\label{it:addPoint} If $\NN\subset \MM$ is complementably $p$-amenable in $\MM$ with constant $K$ then $\NN\cup\{x_0\}$ is complementably $p$-amenable in $\MM$ with constant $KC$.
\end{enumerate}
\end{Lemma}
\begin{proof}
Pick $\varepsilon\in (0,2^{1/p}-1)$. If $x_0\in\MM$ is an isolated point, we pick $0\in\MM\setminus\{x_0\}$ with $d(x_0,0) < (1+\varepsilon)^{1/p}\inf\{d(x_0,x)\colon x\in \MM\setminus\{x_0\}\}$.

\ref{it:omitPoint}
If $x_0\in\MM$ is not an isolated point then $\MM\setminus\{x_0\}$ is complementably $p$-amenable in $\MM$ with constant $1$, otherwise it is easy to see that the map $f\colon\MM\to\F_p(\MM\setminus\{x_0\})$ given by $f(x)=\delta_{\MM\setminus\{x_0\}}(x)$ for $x\in \MM\setminus\{x_0\}$ and $f(x_0)=0$ is $2^{1/p}$-Lipschitz which, using Lemma~\ref{lem:complementedCondition}, gives \ref{it:omitPoint}.

\ref{it:omitPoint1dim} and \ref{it:omitPointIso} If $x_0\in\MM$ is isolated then for all $x\in\MM\setminus\{0\}$ we have
\begin{equation}\label{eq:sum}
d^p(x_0,0)+d^p(x,0)\le d^p(x_0,x)+2d^p(x_0,0)\le (1+2(1+\varepsilon)^p)d^p(x_0,x).
\end{equation}
which, by \cite{AACD2019}*{Lemma 2.1}, implies that $\F_p(\MM)\simeq_{C'}\Rea\oplus \F_p(\MM\setminus\{x_0\})$ for $C'=(1+2(1+\varepsilon)^p)^{1/p}$.

If $x_0\in\MM$ is not an isolated point then $\F_p(\MM)$ is isometric to $\F_p(\MM\setminus\{x_0\})$.

In the case when $\MM$ is a metric space then, by \cite{AACD2019}*{Theorem 3.1}, $\ell_p$ is $D$-complemented in $\F_p(\MM\setminus\{x_0\})$ for every $D>2^{1/p}$ and so $\F_p(\MM\setminus\{x_0\})$ is $(1+2D^p)^{1/p}$-isomorphic to $\ell_p\oplus \F_p(\MM\setminus\{x_0\})$ which easily implies that $\F_p(\MM\setminus\{x_0\})\simeq_{D'} \Rea\oplus\F_p(\MM\setminus\{x_0\})$ for every $D'>(1+2D^p)^{2/p}$.

\ref{it:addPoint} is easy if $x_0\in \overline{\NN}$, so we may assume that this is not the case. By Lemma~\ref{lem:complementedCondition}, there is a $K$-Lipschitz map $f\colon\MM\to\F_p(\NN)$ with $f(x)=\delta_\NN(x)$ for every $x\in\NN$. Consider $f'\colon\MM\to\F_p(\NN\cup\{x_0\})$ defined as $f'(x) = L_{\jmath}(f(x))$ for $x\in\MM\setminus\{x_0\}$, where $L_{\jmath}\colon\F_p(\NN)\to\F_p(\NN\cup\{x_0\})$ is the canonical linear map and $f'(x_0)=\delta_{\NN\cup\{x_0\}}(x_0)$. Now, using \eqref{eq:sum} we readily check that $f'$ is $(K(1+2(1+\varepsilon)^p)^{1/p})$-Lipschitz, and an application of Lemma~\ref{lem:complementedCondition} finishes the proof.
\end{proof}

\subsection{The Approximation Property in quasi-Banach spaces}
A quasi-Banach space $X$ is said to have the \emph{approximation property} (AP for short) if there exists a net $(T_\alpha)_{\alpha\in\A}$ of finite-rank operators on $X$ that converge to the identity map $\Id_X$ uniformly on compact sets. If moreover, the net satisfies
\begin{equation}\label{eq:netbdd}
\liminf_\alpha \Vert T_\alpha\Vert \le \lambda,
\end{equation}
 for some constant $\lambda\in[1,\infty)$, then we say that $X$ has \emph{the $\lambda$-bounded approximation property} ($\lambda$-BAP for short). If, in addition to
 \eqref{eq:netbdd} he operators of the net commute, i.e., $T_\alpha\circ T_\beta=T_\beta\circ T_\alpha$ for all $\alpha$, $\beta\in A$, then $X$ is said to have the \emph{commuting $\lambda$-bounded approximation property} (commuting $\lambda$-BAP for short). In turn, if in addition to \eqref{eq:netbdd}, the operators are projections, i.e., $T_\alpha^2=T_\alpha$ for all $\alpha\in A$, we say that $X$ has the \emph{$\pi_\lambda$-property.} We will refer to $X$ simply as having BAP if it has the $\lambda$-BAP for some $\lambda>0$, and we will say that $X$ has \emph{the metric approximation property} (MAP for short) if it has the $1$-BAP. Similarly, if $X$ has the $\pi_\lambda$-property for some $\lambda$ we say that $X$ has the $\pi$-property, and we say that $X$ is the metric $\pi$-property if it has de $\pi$-property. It is obvious that any finite-dimensional quasi-Banach space has the commuting MAP and the metric $\pi$-property.

A \emph{finite dimensional decomposition} of a (separable) quasi-Banach space $X$ is a sequence $(X_n)_{n=1}^\infty$ of finite-dimensional subspaces of $X$ such that every $x\in X$ has a unique series expansion $x=\sum_{n=1}^\infty x_n$ with $x_n\in X_n$ for all $n\in\Nat$. If $X$ admits afinite dimensional decomposition we say that it has the \emph{finite dimensional decomposition property} (FDD property for short). A quasi-Banach space with the FDD property has both the MAP and the metric $\pi$-property under a suitable renorming.

For background on approximation properties we refer the reader to \cite{CasHandbook}. We would like to point out, however, that \cite{CasHandbook} is written within the framework of Banach spaces and, since the Hahn-Banach theorem is heavily used, the proofs of some of the results therein do not carry out automatically to non-locally convex quasi-Banach spaces. In spite of that initial drawback,  it happens that the dependence on the local convexity of the space
 in the proofs of the most accessible characterizations of the above-mentioned approximation properties  can be circumvented, and so they still hold in non-locally convex spaces. As the definitions that appear in the few papers  that touch the subject (\cites{Kalton1977c,Kalton1984,CCM19}) are not unified, we will state the corresponding characterizations and skip the details in order to not divert too far from the flow of the article. A quasi-Banach space $X$ has the AP if and only if for any compact set $K\subset X$ and every $\varepsilon>0$ there is a finite-rank operator $T\colon X\to X$ with $\max_{x\in K} \Vert x-T(x)\Vert \le \varepsilon$. The space $X$ has the $\lambda$-BAP if and only if for any $\varepsilon>0$ such operators exist with the extra-property that $\Vert T \Vert <\lambda+\varepsilon$. To characterize the $\pi_\lambda$-property we must also impose
each operator $T$ to be a projection; and to characterize commuting $\lambda$-BAP we must impose that the operators belong to a given commuting set. A subset $\A$ of the algebra of bounded linear operators from $X$ into $X$ is said to be \emph{commuting} if $T\circ S=S\circ T$ for all $S$, $T\in \A$. These results yield that if a quasi-Banach space $X$ has the BAP (resp. $\pi$-property) there is a smallest constant $\lambda$ such that $X$ has the $\lambda$-BAP (resp. $\pi_\lambda$-property).

Approximation properties are inherited by infinite sums and complemented subspaces. For further reference we record these in the next two propositions.\begin{Proposition}\label{APP1} Let $(X_n)_{n=1}^\infty$ be a sequence of quasi-Banach spaces with $\sup_n \kappa_2(X_n)<\infty$. Suppose that each $X_n$ has the AP (resp. $\lambda$-BAP, $\pi_\lambda$-property, or commuting $\lambda$-BAP for some $\lambda\ge 1$). Then $\left(\bigoplus_{n=1}^\infty X_n\right)_p$ has the AP (resp. $\lambda$-BAP, $\pi_\lambda$-property, or commuting $\lambda$-BAP) for all $0\le p<\infty$.
\end{Proposition}

\begin{Proposition}\label{APP2} Let $Y$ be a quasi-Banach space with the AP (resp. $\lambda$-BAP for some $\lambda\ge 1$). Suppose that a quasi-Banach space $X$ is $K$-complemented in $Y$. Then $X$ has the AP (resp. $\lambda K$-BAP).
\end{Proposition}

\section{General techniques}\label{sect:generalEllPSum}\noindent
In \cite{Kalton2004}*{$\S$4}, Kalton defined an operator $T$ of norm $72$ which maps any Lipschitz-free space $\F(\MM)$ over a metric space $\MM$ into the $\ell_1$-sum of Lipschitz-free spaces over annuli in $\MM$. This particular operator has seen many applications in the theory of Lipschitz-free spaces. Let us single out the most significant ones.
\begin{enumerate}[label={$\bullet$},leftmargin=*]
\item If $\MM$ is a uniformly separated metric space, then $\F(\MM)$ is a Schur space with the Radon-Nikodym property and the approximation property (\cite{Kalton2004}*{Proposition 4.3}).
\item If $\MM$ is a countable compact metric space (or even a countable metric space whose closed balls are compact) then $\F(\MM)$ has the metric approximation property (\cites{Da15, D15}) and the Schur property (\cite{HLP16}).
\item If $X$ is a Banach space then $\F(X)\simeq\ell_1(\F(X))\simeq \F(B_X)$ (\cite{K15}).
\end{enumerate}
This section is geared towards developing an extended version of this operator that works both for $p<1$ and also for the case $p=1$. This is the subject of Lemma~\ref{lem:KaltonAnologyGeneralized}. Our main outcomes will be a couple of complementability results (namely, Theorem~\ref{thm:complemented} and Theorem~\ref{thm:reversecomplementedGerenalization}) as well as several general conditions under which Lipschitz free $p$-spaces over metric spaces are isomorphic to their $\ell_p$-sums (Theorem~\ref{thm:sumisomorphimGeneralized}, Theorem~\ref{thm:sumisomorphimGeneralized:1}, and Theorem~\ref{thm:sumisomorphimGeneralized:2}). In subsequent sections we will concentrate on new applications to the general theory, some of which were not known even for the case $p=1$.

Let $(\MM,d)$ be a $p$-metric space, $0<p\le 1$. If $(\MM_\alpha)_{\alpha\in\Delta}$ is a family of subsets of $\MM$ with $0\in\MM_\alpha$ for all $\alpha\in\Delta$, we can define a norm-one linear operator
\begin{equation}\label{eq:5}
P\colon \left( \bigoplus_{\alpha\in\Delta}\F_p(\MM_{\alpha})\right)_p \to \F_p(\MM),\quad (\mu_\alpha)_{\alpha\in\Delta} \mapsto \sum_{\alpha\in\Delta} L_\alpha(\mu_\alpha),
\end{equation}
where $L_\alpha$ is the canonical linear embedding from $\F_p(\MM_\alpha)$ into $\F_p(\MM)$. We aim to relate the Lipschitz free space over a $p$-metric space $\MM$ to direct sums of Lipschitz free $p$-spaces over subsets of $\MM$ by means of such mappings. Our first proposition is inspired by \cite{Kalton2004}*{$\S$4}.

\begin{Proposition}\label{lem:KaltonAnalogyBis}Let $(\MM,d,0)$ be a pointed $p$-metric space, $0<p\le 1$. Suppose that $N$ is a countable subset of $\Int$ and that
$(A_n)_{n\in N}$ is a sequence of positive real numbers such that $d(0,x)\in\cup_{n\in N} A_n$ for all $x\in \MM\setminus\{0\}$, and
\[
K=\inf_{m<n}\frac{\inf A_n}{\sup A_m}>1.
\]
Then the operator $P$ defined as in \eqref{eq:5} corresponding to the family $(\MM_{A_n}^*)_{n\in N}$ is an isomorphism. In fact,
\[
\Vert P^{-1}\Vert \le (K^{p}+1)^{1/p} (K^p-1)^{-1/p}.\]
\end{Proposition}

\begin{proof} For $n\in N$ put $r_n=\inf A_n$ and $s_n=\sup A_n$. If $m<n$, $x\in \MM_{A_m}$, and $y\in \MM_{A_n}$ we have
\[
d^p(0,x)\le s_m^p=\frac{s_m^p}{r_n^p} r_n^p \le K^{-p} d^p(0,y).
\]
Hence, if we denote $d^p(0,y)/d^p(0,x)=t$,
\[
d^p(0,x) + d^p(0,y)=(d^p(0,y)-d^p(0,x)) \frac{t+1}{t-1}\le d^p(x,y) \frac{K^{p}+1}{K^p-1}.
\]
Since $(\MM_{A_n})_{n\in N}$ is a partition of $\MM\setminus\{0\}$, the result follows from \cite{AACD2019}*{Lemma 2.1}.
\end{proof}

Proposition~\ref{lem:KaltonAnalogyBis} leads us to consider maps as in \eqref{eq:5} corresponding to families of the form
$(\MM_{A_n}^*)_{n\in\Int}$ with $A_n=[cR^{2n},cR^{2n+1}]$ for some $0<c<\infty$ and $R>1$. However, when $p<1$ it is unknown whether the canonical linear map from $\F_p(\cup_{n\in\Int} \MM_{A_n}^*)$ into $\F_p(\MM)$ is an isomorphic embedding, therefore in order to obtain useful information about this kind of mappings we need to develop complementary techniques. Our approach here will consist of building, under suitable conditions on the $p$-metric space $\MM$, both left and right inverses of operators as in \eqref{eq:5} corresponding to families of the form $(\MM_{A_n}^*)_{n\in\Int}$ with $A_n= [cR^{sn},cR^{sn+1}]$ for some $0<c,s<\infty$ and $R>1$.
\begin{Definition} A family $\UU$ of open sets in a topological space $X$ is said to be \emph{$k$-overlapping}
if each $x\in X$ belongs at most to $k$ members of $\UU$. We say that $\UU$ is a \emph{point-finite family} if it is $k$-overlapping for some $k\in\Nat$.
\end{Definition}

\begin{Lemma}\label{lem:KaltonAnologyGeneralized}
Suppose $(\MM,d,0)$ is a pointed metric space. Let $k\in\Nat$, $R>1$, $K_1>0$, and $K_2>0$ be constants such that:
\begin{enumerate}[label={$\bullet$}]
\item $(\psi_n)_{n\in N}$ is a countable family of $K_1$-Lipschitz maps from $\Rea$ into $\Rea$ which are uniformly bounded by $K_2$; and
\item $(I_n)_{n\in N}:=\big(\psi_n^{-1}(\Rea\setminus\{0\})\big)_{n\in N}$ is $k$-overlapping.
\end{enumerate}
Then for each $0<p\le 1$ there is a bounded linear operator
\[
T \colon \F_p(\MM)\to X:=\left( \bigoplus_{n\in N}\F_p\left(\MM^*_{R^{I_n}}\right)\right)_p
\]
with
\begin{equation}\label{eq:7}
T(\delta_\MM(x))=(\psi_n(\log_R d(0,x)) \delta_n(x))_{n\in N},\quad x\in\MM\setminus\{0\},
\end{equation}
where $\delta_n\colon \MM\to\F_p(\MM^*_{R^{I_n}})$ is any extension of the canonical embedding of $\MM^*_{R^{I_n}}$ into the Lipschitz free $p$-space over $\MM^*_{R^{I_n}}$. Moreover, $\|T\|\leq C$ for some $C=C(p,k,R,K_1,K_2)$.
\end{Lemma}
\begin{proof} For $n\in N$, put $\alpha_n(x)=\psi_n(\log_R d(0,x))$ if $x\in\MM\setminus\{0\}$ and $\alpha_n(0)=0$. Without loss of generality we assume that $\delta_n(x)=0$ if $x\notin\MM_{R^{I_n}}$. Define $f\colon \MM \to X$ by
\[
f(x)=(f_n(x))_{n\in N}:=(\alpha_n(x) \, \delta_n(x))_{n\in N}.
\]

Let $x$, $y\in\MM$ and assume without loss of generality that $d(0,y)\le d(0,x)$. We have
\[
\Vert f(y)-f(x)\Vert \le (2k)^{1/p} \sup_{n\in N} \Vert f_n(y)-f_n(x)\Vert.
\]
The $p$-subadditivity of the quasi-norm gives
\begin{align*}
\Vert f_n(y)&-f_n(x)\Vert^p
= \Vert (\alpha_n(y) -\alpha_n(x))\delta_n(y)+\alpha_n(x)(\delta_n(y)-\delta_n(x))\Vert^p \\
&\le |\alpha_n(y) -\alpha_n(x)|^p \, \Vert\delta_n(y)\Vert^p+|\alpha_n(x)|^p \Vert \delta_n(y)-\delta_n(x)\Vert^p\\
&\le |\alpha_n(y) -\alpha_n(x)|^p d^p(0,y)+|\alpha_n(x)|^p \Vert \delta_n(y)-\delta_n(x)\Vert^p.
\end{align*}
Taking into account that $\log u \le u-1$ for every $u\ge 1$ we obtain
\[
|\alpha_n(y) -\alpha_n(x)| \, d(0,y)\le K_1 d(0,y) \log_R \frac{ d(0,x)}{d(0,y)}\le K_1\frac{d(0,x)-d(0,y)}{\log R}.
\]
Since $|\alpha_n(x)|\le K_2$, if $x\notin\MM_{R^{I_n}}$, or $\{x,y\}\subseteq \MM_{R^{I_n}}$ we have
\[
|\alpha_n(x)| \, \Vert \delta_n(y)-\delta_n(x)\Vert \le |\alpha_n(x)| \,d(x,y) \le K_2 d(x,y).
\]
Assume that $x\in\MM_{R^{I_n}}$ and $y\notin\MM_{R^{I_n}}$. If $d(0,x)\leq Rd(0,y)$ we have
\begin{align*}
|\alpha_n(x)| \, \Vert \delta_n(y)-\delta_n(x)\Vert &= |\alpha_n(x)-\alpha_n(y)| \, d(0,x) \\
&\le R |\alpha_n(x)-\alpha_n(y)| d(0,y),
\end{align*}
and if $d(0,x)\geq Rd(0,y)$,
\begin{align*}
|\alpha_n(x)| \, \Vert \delta_n(y)-\delta_n(x)\Vert &= |\alpha_n(x)| \, \frac{ Rd(0,x) - d(0,x)}{R-1}\\
&\le |\alpha_n(x)| R \frac{d(0,x)-d(0,y)} {R-1}\\
&\le
\frac{K_ 2 R}{R-1} (d(0,x)-d(0,y)).
\end{align*}
Summing up, since $ d(0,x)-d(0,y)\le d(x,y)$, we get that $f$ is $C$-Lipschitz for
\begin{equation}\label{NormofT}
C=(2k)^{1/p}
\left(\frac{K_1^p}{\log^p R}
+\max\left\{ \frac{K_1^p R^p}{\log^p R}, \frac{K_2^p R^p}{(R-1)^p} \right\}
\right)^{1/p}.
\end{equation}
Thus, by \cite{AACD2018}*{Theorem 4.5}, there is a linear map $T\colon \F_p(\MM) \to X$ such that $T\circ\delta_\MM=f$ and $\Vert T\Vert \le C$.
\end{proof}

We use $[a,b]$ for a closed real interval, and if either $a$ or $b$ are infinity, we replace the corresponding bracket with a parenthesis.
\begin{Lemma}\label{lem:partitionunity}
Let $k\in\Nat$ and $r>0$. There is a constant $C=C(k,r)$ such that whenever $(a_n,b_n)_{n\in N}$ is a $k$-overlapping family of open intervals in $\Rea$ with
\[
\Rea = \bigcup_{n\in N} [a_n+r,b_n-r],
\]
then there is a family $(\psi_n)_{n\in N}$ of $C$-Lipschitz functions from $\Rea$ into $[0,1]$ such that
\[
\sum_{n\in N} \psi_n(u)=1, \quad u\in\Rea,
\]
and, for $n\in N$, $\psi_n(u)>0$ if and only if $u\in(a_n,b_n)$.
\end{Lemma}
\begin{proof}For each $n\in N$ there is a $1$-Lipschitz piecewise linear function $\phi_n\colon\Rea\to[0,r]$ such that $\phi_n(u)=r$ if $u\in[a_n+r,b_n-r]$, and $\phi_n(u)=0$ if and only if $u\notin(a_n,b_n)$. The map
\[
\Phi(u)=\sum_{n\in N} \phi_n(u), \quad u\in\Rea
\]
is $2k$-Lipschitz, and $\Phi(u)\in[r,kr]$ for every $u\in\Rea$. Set
\[
\psi_n=\frac{\phi_n}{\Phi}, \quad n\in N.
\]
We have $\psi_n(u)\in[0,1]$ for all $n\in N$ and all $u\in\Rea$. Let $n\in N$ and $u$, $v\in\Rea$. We have
\begin{align*}
|\psi_n(u)-\psi_n(v)|&=\frac{|\phi_n(u)\Phi(v) -\phi_n(u)\Phi(x)+\phi_n(u)\Phi(u) -\phi_n(v)\Phi(u) |}{\Phi(u)\Phi(v)}\\
&\le\frac{\phi_n(u)|\Phi(v) -\Phi(u)|+|\phi_n(u) -\phi_n(v)|\Phi(u)}{\Phi(u)\Phi(v)}\\
&\le \frac{2kr+kr}{r^2}|u-v|\\
&=\frac{3k}{r}|u-v|.\qedhere
\end{align*}
\end{proof}

\begin{Theorem}\label{thm:complemented} Suppose $(\MM,d,0)$ is a pointed metric space. For some constants $k\in\Nat$ and $r>0$ let
$(I_n)_{n\in N}$ be a $k$-overlapping countable family of sets in $\Rea$ for which there is $([a_n,b_n])_{n\in N}$with $(-r+a_n,r+b_n)\subseteq I_n$ for all $n\in N$. Then if $R>1$ we have
\[
\F_p(\MM) \unlhd_C \bigoplus_{n\in N} \F_p\left(\MM^*_{R^{I_n }}\right)
\]
for all $0<p\le 1$, where $C=C(p,k,R,r)$ (in particular, the value of the constant $C$ does not depend on the metric space $\MM$).
\end{Theorem}

\begin{proof}
Let $(\psi_n)_{n\in N}$ be the family of Lipschitz functions whose existence is guaranteed by Lemma~\ref{lem:partitionunity} with respect to the intervals $(J_n)_{n\in N}$, where $J_n=(-r+a_n,r+b_n)$. Consider the operator
\[
T\colon \F_p(\MM)\to Y:=\left( \bigoplus_{n\in N}\F_p(\MM^*_{R^{J_n}})\right)_p
\]
provided by Lemma~\ref{lem:KaltonAnologyGeneralized}. Let $S$ be the canonical embedding of $Y$ into $\bigoplus_{n\in N} \F_p\left(\MM^*_{R^{I_n }}\right)$. Let $P$ be the operator defined as in \eqref{eq:5}. For $x\in\MM$ we have
\[
P(S(T(\delta(x))))
=\left(\sum_{n\in N}\psi_n(\log_R d(0,x))\right)\delta(x)
=\delta(x).
\]
By linearity and continuity, $P\circ S \circ T =\Id_{\F_p(\MM)}$.
\end{proof}

\begin{Theorem}\label{thm:reversecomplementedGerenalization} Let $(\MM,d,0)$ be a pointed metric space. Suppose that \begin{enumerate}[label={$\bullet$}]
\item $(I_n)_{n\in N}$ is a countable sequence of subsets of $\Rea$ for which there is a pairwise disjoint family $(J_n)_{n\in N}$ of open intervals such that if $J_n=(a_n,b_n)$, then $I_n\subseteq [a_n+r,b_n-r]$ for some $r>0$;
\item There is $R>1$ such that $\MM^*_{R^{I_n}}$ is complementably $p$-amenable in $\MM^*_{R^{J_n}}$ with constant $K>1$ for all $n\in N$.
\end{enumerate} Then via the canonical operator defined in \eqref{eq:5}we have
\[
\left( \bigoplus_{n\in N}\F_p(\MM^*_{R^{I_n}})\right)_p \unlhd_C \F_p(\MM)
\]
for all $0<p\le 1$, where $C=C(p,r,R,K)$ (in particular, the value of the constant $C$ does not depend on the metric space $\MM$).
\end{Theorem}
\begin{proof}
For $n\in N$ pick a $1$-Lipschitz function $\psi_n\colon\Rea\to [0,r]$ such that
$\psi(u)=1$ if $u\in[a_n+r,b_n-r]$ and $\psi_n(u)=0$ if $u\notin J_n$.
Let
\[
T\colon \F_p(\MM)\to \left( \bigoplus_{n\in Z}\F_p(\MM^*_{R^{J_n}})\right)_p
\]
be the linear map defined as in \eqref{eq:7}. By Lemma~\ref{lem:KaltonAnologyGeneralized}, $T$ is bounded. Let $P$ be the operator defined as in \eqref{eq:5} corresponding to the family $(\MM^*_{R^{I_n}})_{n\in N}$. By the assumption, there are linear maps $(E_n)_{n\in N}$ from $\F_p(\MM^*_{R^{J_n}})$ into $\F_p(\MM^*_{R^{I_n}})$ with $\|E_n\| \leq K$ and $E_n\circ L_{\jmath}=\Id$. Then the linear map
\[
E\colon \left( \bigoplus_{n\in N}\F_p(\MM^*_{R^{J_n}})\right)_p \to X:=\left( \bigoplus_{n\in N}\F_p(\MM^*_{R^{I_n}})\right)_p
\]
defined by $E=(E_n)_{n\in Z}$ is bounded by $K$. Denote by $\delta_n$ the $\delta$-function in $\MM^*_{R^{I_n}}$. Denote by $\II_k(h)$ the sequence $(h_n)_{n\in N}$ defined by $h_k=h$ and $h_n=0$ is $n\not=k$. If $x\in\MM_{R^{I_n}}$, taking into account that $\psi_n(u)=1$ if $u\in I_n$ and that $\psi_n(u)=0$ if $u\notin J_n$ we obtain
\[
\II_n(\delta_n(x))
\stackrel{P}\mapsto \delta_\MM(x)
\stackrel{T}\mapsto
\II_{n}(L_{\jmath}(\delta_{n}(x)))
\stackrel{E}\mapsto \II_{n}(\delta_{n}(x)).
\]
Therefore $E\circ T \circ P=\Id_{X}$.
\end{proof}

Since we are dealing with complementability relations between quasi-Banach spaces, Pe{\l}czy\'nski's decomposition method (see, e.g., \cite{AlbiacKalton2016}*{Theorem 2.2.3}) will be a key ingredient in our arguments. In particular, the following lemma will be frequently used.
\begin{Lemma}\label{lem:InfiniteSums} Let $(J,\le)$ be a partially ordered set, let $1\le K<\infty$ be a constant and let $(X_j)_{j\in J}$ be a family of Banach spaces such that $X_i$ is $K$-complemented in $X_j$ whenever $i\le j$. For $i=1$, $2$, let $N_i$ be a countable set, and let $\phi_i\colon N_i\to J$ be such that for every $j\in J$ and $F\subset N_i$ finite there is $n\in N_i\setminus F$ with $j\le \phi_i(n)$. For each $p\in[0,\infty]$ set
\[
X_i=\left(\bigoplus_{n\in N_i} X_{\phi_i(n)} \right)_p, \quad i=1,2.
\]
(with the convention that $\ell_p(X)$ means $c_0(X)$ if $p=0$). Then there is a constant $C$ depending only on $K$ and $p$ such that $X_{1}\simeq_C X_{2}\simeq_C \ell_p(X_{1})$.
\end{Lemma}

\begin{proof} Let $\OO$ be the set of all sequences $\phi\colon \Nat\to J$ such that for every
$j\in J$ and every $k\in\Nat$ there is $n\in N$ with $k<n$ and $j\le \phi(n)$. For $\phi\in\OO$ put
\[
X_{\phi}=\left(\bigoplus_{n=1}^\infty X_{\phi(n)} \right)_p.
\]
We must prove that $X_\phi\simeq_C X_\psi\simeq \ell_p(X_\phi)$ for all $\phi$, $\psi\in\OO$ and some $C=C(p,K)$.
For $j\in J$ and $\phi\in\OO$ put
\[
N(\phi,j)=|\{n\in \Nat \colon \phi(n)=j\}|.
\]
The symmetry of the norm of $\ell_p$ yields that:
\begin{enumerate}[label={(\roman*)},widest=iii]
\item\label{DM1} $X_{\phi}\simeq X_{\psi}$ isometrically if $N(\phi,j)=N(\psi,j)$ for all $j\in J$; and
\item\label{DM2} $X_{\phi}$ is $1$-complemented in $X_{\psi}$ if $N(\phi,j)\le N(\psi,j)$ for all $j\in J$.
\end{enumerate}
Using \ref{DM1} and the fact that an iteration of an $\ell_p$-norm is an $\ell_p$-norm we have:
\begin{enumerate}[resume*]
\item\label{DM3} $\ell_p(X_{\psi})\simeq X_{\psi}$ isometrically if $N(\psi,j)=\infty$ for all $j\in J$.
\end{enumerate}
From our assumption we readily infer:
\begin{enumerate}[resume*]
\item\label{DM4} $X_{\phi}$ is $K$-complemented in $X_{\psi}$ if $\phi(n)\le\psi(n)$ for all $n\in \Nat$.
\end{enumerate}
Pick $\psi\in\OO$ such that $|N(\psi,j)|=\infty$ for all $j\in J$. Let $\phi\in\OO$. By \ref{DM3}, $\ell_p(X_{\psi})\simeq X_{\psi}$ isometrically. By \ref{DM2}, $X_{\phi}$ is $1$-complemented in $X_{\psi}$. The definition of $\OO$ yields the existence of an increasing sequence $(k_n)_{n=1}^\infty$ and a map $\alpha\colon N\to N$ such that $\psi(n) \le \rho(n)=\phi(k_n)$ for all $n\in\Nat$. By \ref{DM2} and \ref{DM4}, $X_{\psi}$ is $K$-complemented in $X_{\rho}$, which in turn is $1$-complemented in $X_{\phi}$. By Pe{\l}czy\'nski's decomposition method, $X_{\phi}\simeq_C X_{\psi}$ for some constant $C$ depending on $K$ and $p$.
\end{proof}

Next we see a few instances of situations where we may combine the results from this subsection.
\begin{Theorem}\label{thm:sumisomorphimGeneralized}Suppose $\MM$ is an infinite metric space.
Let $0<p\le 1$, $R>1$, $\lambda$, $\mu$, $K_1$, $K_2$, $K_3>1$ and $\alpha>0$. There is a constant $C$ such that, if $A_n=(c_1R^n, c_2 R^n]$ for $n\in\Nat_*$, then $\F_p(\MM)\simeq_C \ell_p (\F_p(\MM))$ and
\[
\F_p(\MM)
\simeq_C \left(\bigoplus_{n\in N} \F_p\left(\MM^*_{A_{\phi(n)}} \right)\right)_p
\simeq_C \left(\bigoplus_{n\in N} \F_p\left(\MM_{A_{\phi(n)}} \right)\right)_p,
\]
whenever $N$ is a countable set, $\phi\colon N\to\Nat_*$ is unbounded, $0<c_1<c_3<c_2<\infty$ satisfies $\log_R (c_2/c_1)>\lambda$ and $\log_R (c_3/c_1)>\alpha$, and $(\MM,d,0)$ is an infinite pointed metric space such that
\begin{enumerate}[label={$\bullet$}]

\item $d(0,x)> c_1$ for all $x\in\MM\setminus\{0\}$;

\item $\F_p(\MM_{A_n})$ is $K_1$-complemented in $\F_p(\MM_{A_m})$ for all $n,m\in \Int$ with $0\le n\leq m$;

\item $\F_p(\MM_{(c_1, c_3]})$ is $K_2$-complemented in $\F_p(\MM_{A_0})$; and

\item If $B_n=( c_1\mu^{-1}R^n, c_2\mu R^n]$, then, for $n$ large enough,
$\MM_{A_{n}}$ is complementably $p$-amenable in $\MM_{B_n}$ with constant $K_3$.
\end{enumerate}
\end{Theorem}

\begin{Theorem}\label{thm:sumisomorphimGeneralized:1}
Suppose $\MM$ is an infinite metric space. Let $0<p\le 1$, $R>1$, $\lambda$, $\mu$, $K_1$, $K_2$, $K_3>1$ and $\alpha>0$. There is a constant $C$ such that $\F_p(\MM)\simeq_C \ell_p (\F_p(\MM))$ and, if $A_n=(c_1R^{-n}, c_2 R^{-n}]$ for $n\in\Nat_*$, then
\[
\F_p(\MM)
\simeq_C \left(\bigoplus_{n\in N} \F_p\left(\MM^*_{A_{\phi(n)}} \right)\right)_p
\simeq_C \left(\bigoplus_{n\in N} \F_p\left(\MM_{A_{\phi(n)}} \right)\right)_p,
\]
whenever $N$ is a countable set, $\phi\colon N\to\Nat_*$ is unbounded, $0<c_1<c_3<c_2<\infty$ satisfy $\log_R (c_2/c_1)>\lambda$, and $\log_R (c_3/c_1)>\alpha$, and $(\MM,d,0)$ is an infinite pointed metric space such that
\begin{enumerate}[label={$\bullet$}]

\item $d(0,x)\le c_2$ for all $x\in\MM\setminus\{0\}$;

\item $\F_p(\MM_{A_n})$ is $K_1$-complemented in $\F_p(\MM_{A_m})$ for all $n,m\in \Int$ with $0\le n\leq m$;
\item $\F_p(\MM_{(c_3, c_2]})$ is $K_2$-complemented in $\F_p(\MM_{A_0})$; and

\item If $B_n=( c_1\mu^{-1}R^{-n}, c_2\mu R^{-n}]$, then
$\MM_{A_{n}}$ is complementably $p$-amenable in $\MM_{B_n}$ with constant $K_3$ for $n$ large enough.
\end{enumerate}
\end{Theorem}

\begin{Theorem}\label{thm:sumisomorphimGeneralized:2}
Let $(\MM, d,0)$ be an infinite pointed metric space.
Let $0<p\le 1$. Given constants $R>1$, $\lambda>1$, $\mu>1$, $K_1>1$, and $K_2>1$ there is a constant $C=C(p,R,\lambda,\mu,K_1,K_2)$ such that
\[
\F_p(\MM)
\simeq_C \ell_p (\F_p(\MM))
\simeq_C \ell_p (\F_p(\MM^*_{(c_1,c_2]}))
\simeq_C \ell_p (\F_p(\MM_{(c_1,c_2]}))
\]
whenever $c_1>0$ and $c_2>0$ satisfy $\log_R(c_2/c_1)\ge\lambda$, and\begin{enumerate}[label={$\bullet$}]
\item if $A_n=(c_1R^n, c_2R^n]$, then
$\F_p(\MM_{A_n})\simeq_{K_1}\F_p(\MM_{A_0})$ for all $n\in \Int$; and

\item if $B_n=(c_1\mu^{-1}R^n, c_2 \mu R^{n})$, then $\MM_{A_{n}}$ is complementably $p$-amenable in $\MM_{B_n}$ with constant $K_2$ for all $n\in\Int$.
\end{enumerate}
\end{Theorem}

\begin{proof}[Proof of Theorems~\ref{thm:sumisomorphimGeneralized}, \ref{thm:sumisomorphimGeneralized:1} and \ref{thm:sumisomorphimGeneralized:2}] As the proofs of Theorem~\ref{thm:sumisomorphimGeneralized:2} and \ref{thm:sumisomorphimGeneralized:1} are similar, we shall only take care of the proof of Theorem~\ref{thm:sumisomorphimGeneralized}.

By Lemma~\ref{lem:separatedpoint} we have (for possibly larger constants $K_1$, $K_2$ and $K_3$) that $\F_p(\MM_{A_n}^*)$ is $K_1$-complemented in $\F_p(\MM_{A_m}^*)$, that $\F_p(\MM_{(c_1,c_3]}^*)$ is $K_2$-complemented if $\F_p(\MM_{A_0}^*)$, and that $\MM^*_{A_{n+2}}$ is complementably $p$-amenable in $\MM^*_{B_n}$ with constant $K_3$.

For $N$ countable and $\phi\colon N\to\Nat_*$ unbounded, put
\[
X_\phi:=\left(\bigoplus_{n\in N} \F_p(\MM^*_{A_{\phi(n)}})\right)_p.\] Let $N_1=\{n\in\Int \colon n\ge -1\}$ and $\phi_1\colon N_1\to \Nat_*$ be defined $\phi_1(n)=\max\{n,0\}$.
Applying Theorem~\ref{thm:complemented} with the family of intervals
\[
\{(-\infty, c_3]\}\cup\{ (\log_R(c_1) +n,\log_R(c_2) +n] \colon n\in\Nat_* \}
\]
we obtain $\F_p(\MM)\unlhd_{C_1} X_{\phi_1}$ for some constant $C_1$ depending on $p$, $\lambda$, $\alpha$, $R$ and $K_2$.

For $s\in\Nat$ large enough the intervals
\[
(-\log_R(\mu)+\log_R(c_1) +sn,\log_R(\mu)+\log_R(c_2) +sn), \quad n\in\Nat_*
\]
are mutually disjoint. Set $N_2=\{n\in\Int\colon n\ge n_0\}$ for a suitable $n_0\in\Nat$ and define $\phi_2\colon N_2\to\Nat_*$ by $\phi_2(n)=sn$.
Applying Theorem~\ref{thm:reversecomplementedGerenalization} with the family of intervals
\[
I_n=(\log_R(c_1) +sn,\log_R(c_2) +sn], \quad n\in\Nat, \, n\ge n_0,
\]
we obtain $ X_{\phi_2}\unlhd_{C_2}\F_p(\MM)$ for some constant $C_2$ depending on $p$, $\mu$ $R$, $K_1$ and $K_3$.

An appeal to Lemma~\ref{lem:InfiniteSums}, Lemma~\ref{lem:separatedpoint} and \cite{AACD2019}*{Theorem 3.2} completes the proof.
\end{proof}

\begin{Remark} Analogous results to Theorems~\ref{thm:sumisomorphimGeneralized}, \ref{thm:sumisomorphimGeneralized:1}, and \ref{thm:sumisomorphimGeneralized:2} hold replacing left-open and right-closed intervals with open, closed, or right-open and left-closed intervals.
\end{Remark}

\section{Applications}\label{sec:appl}\noindent
This section is devoted to applications of the techniques we developed in the preceding section to advance the state of art of the general theory even in the case $p=1$.
\subsection{Applications to metric spaces with few limit points}
Here we gather our first applications. Even though some of these applications were not explicitly stated for $p=1$, they follow from known results (see, e.g., Proposition~\ref{prop:FpSubsym}).

\begin{Proposition}\label{prop:finiteAcPoints} Suppose that $\MM$ is a complete metric space with finitely many accumulation points. Then there is a countable family $(\MM_n)_{n\in N}$ of closed bounded subsets of $\MM$ with no accumulation points such that
\[
\F_p(\MM)\unlhd \left(\bigoplus_{n=1}^\infty \F_p(\MM_n)\right)_p,
\]
for all $0<p\le 1$.\end{Proposition}

\begin{proof}
Let $k$ be the number of accumulation points of $\MM$. We proceed by induction on $k$. We prove the case $k=0$ and the general case simultaneously. In the case when $k\ge 1$ we choose as base point of $\MM$ an accumulation point of $\MM$. In the case when $k=0$ we choose as base point an arbitrary point of $\MM$.
Let $c>0$ and $R>1$ be such that all the accumulation points except $0$ are in $\MM_{(cR,cR^2)}$.
Applying Theorem~\ref{thm:complemented} with $I_n=[n+\log_R c,n+2+\log_R c]$ for $n\in\Int$, yields
\[
\F_p(\MM)\unlhd \left(\bigoplus_{n\in\Int} \F_p(\MM_n)\right)_p,
\]
where $\MM_n$ is bounded and closed for every $n\in \Int $, $\MM_n$ has no accumulation points if $n\notin\{0,1\}$, and $\MM_n$ has at most $\max\{ k-1, 0\}$ accumulation points if $n\in\{0,1\}$. The case $k=0$ is finished. In the general case, we may now apply the induction hypothesis for $k-1$.
\end{proof}

\begin{Corollary}\label{cor:lpoflp}Let $\MM$ be a uniformly separated metric space. Then for each $0<p\le 1$ there are $p$-Banach spaces $(X_n)_{n=1}^\infty$ such that $\F_p(\MM)\unlhd (\oplus_{n=1}^\infty X_n)_p$ and $X_n\simeq \ell_p(I)$ for all $n\in\Nat$, where $I=|\MM|$.
\end{Corollary}

\begin{proof}By Proposition~\ref{prop:finiteAcPoints},
$\F_p(\MM)\unlhd \left(\bigoplus_{n\in\Int} \F_p(\MM_n)\right)_p$, where each metric space $\MM_n$ is bounded and uniformly separated. Then, by \cite{AACD2018}*{Theorem 4.14}, $\F_p(\MM_n)\simeq \ell_p(|\MM_n|-1)$ (non necessarily uniformly).
\end{proof}

\begin{Corollary}\label{cor:unifDiscr}Let $\MM$ be a uniformly separated metric space. Then $\F_p(\MM)$ has the AP for each $0<p\le 1$.
\end{Corollary}
\begin{proof} The statement follows from Corollary~\ref{cor:lpoflp} taking into account that $\ell_p$ has the AP, and that AP is preserved by $\ell_p$-sums and by complemented subspaces (see Propositions~\ref{APP1} and \ref{APP2}).
\end{proof}

\begin{Proposition}\label{prop:finiteCoronas}Let $\MM$ be a metric space. Suppose that there is a point $x\in\MM$ such that every annulus centered at $x$ has only finitely many points. Then for each $0<p\le 1$ there is a countable family $(X_n)_{n\in N}$ of finite-dimensional $p$-Banach spaces such that $\F_p(\MM)\unlhd (\oplus_{n\in N} X_n)_p$.
\end{Proposition}
\begin{proof}
Just apply Theorem~\ref{thm:complemented} with $I_n=[n,n+2]$ for $n\in\Int$.
\end{proof}

Since finite-dimensional spaces have the MAP, it follows from Proposition~\ref{prop:finiteCoronas} that if $\MM$ is a metric space with finite annuli, then $\F_p(\MM)$ has BAP. We can improve this result by squeezing a bit more our techniques.

\begin{Proposition}\label{prop:commutingBAP} Suppose $\MM$ is a metric space with a point $x\in\MM$ such that every annulus centered at $x$ contains only finitely many points. Then for each $0<p\le 1$, the space $\F_p(\MM)$ has the $C$-commuting BAP for every $C>4^{1/p}$.
\end{Proposition}
\begin{proof}
The proof relies on an enhancement of the proof of Theorem~\ref{thm:complemented}. Pick $R>1$ and let $(\psi_n)_{n\in \Int}$ be the family of $1/R$-Lipschitz maps given by
\[
\psi_n(x)=\max\left\{1-\frac{|x-Rn|}{R},0\right\}, \quad x\in\Rea.
\]
For $n\in\Int$ consider the intervals $I_n=[Rn-R,Rn+R]$ and $J_n=(Rn-R,Rn+R)$. Define $T$ and $S$ as in the proof of Theorem~\ref{thm:complemented} and set
\begin{equation*}
P_m\colon \left( \bigoplus_{n\in\Int}\F_p(\MM^*_{R^{I_n}})\right)_p \to \F_p(\MM),\, (\mu_n)_{n\in\Int} \mapsto \sum_{n=-m}^m L_n(\mu_n), \, m\in\Nat.
\end{equation*}
The sequence of operators $(S_m)_{m=1}^\infty$ defined by $S_m=P_m\circ S\circ T$ satisfies $\sup_m\Vert S_m\Vert \le \Vert T\Vert$. Since $\sum_{n\in\Int} \psi_n =1$, it follows that $\lim_m S_m=S$ in the strong topology of operators. Each $S_m$ has finite rank, and by construction,
\[
S_m(\delta(x)) = \sum_{n=-m}^m \psi_k(\log_R d(0,x))\delta(x),\quad x\in\MM.
\]
We infer that $S_m\circ S_{m'}= S_{m'}\circ S_m=S_{\min\{m,m'\}}$ for all $m$, $m'\in\Nat$. Finally, note that $\Vert T\Vert\le C$, where $C$ is as in \eqref{NormofT} with $k=2$ and $K_1=1/R$ and $K_2=1$. Letting $R$ tend to $\infty$ we obtain the desired estimate for the commuting BAP constant.
\end{proof}

A basic sequence $(x_j)_{j=1}^\infty$ in a quasi-Banach space $X$ is said to be \emph{subsymmetric} if it is unconditional and equivalent to all its subsequences. Every subsymmetric basic sequence is semi-normalized, i.e.,
\[
\inf_{j\in\Nat} \Vert x_j\Vert>0 \text{ and }
\sup_{j\in\Nat} \Vert x_j \Vert<\infty.
\]
For some more details concerning subsymmetric bases we refer the reader to \cite{Singer1970}*{Chapter II, $\S$21} and \cite{Ansorena2018}. We would like to point out that, although the paper \cite{Ansorena2018} is written within the framework of Banach spaces, it can be re-written verbatim for $p$-Banach spaces without altering the validity of the results.

\begin{Lemma}\label{lem:sy}Let $(X_n)_{n=1}^\infty$ be a sequence of finite dimensional quasi-Banach spaces . Every subsymmetric basic sequence $\XX=(x_j)_{j=1}^\infty$ of $X=(\bigoplus_{n=1}^\infty X_n)_p$, $0\le p<\infty$, is equivalent to the unit vector system of $\ell_p$ (with the convention that $\ell_0$ means $c_0$).
\end{Lemma}
\begin{proof}For $x\in X$, write $x=(x(n))_{n=1}^\infty$ and define
\[
\supp(x)=\{n\in\Nat \colon x(n)\not=0\}.
\]
We also define, for $N\in\Nat$, $S_N(x)\in X$ by $S_N(x)(n)=x(n)$ for $n\le N$ and
$S_N(x)(n)=0$ for $n>N$. Note that $\lim_N S_N(x)=x$.
Since $\XX$ is semi-normalized, $\sup_{j} \Vert x_j(n) \Vert<\infty$ for all $n\in\Nat$. Since $B_{X_n}$ is compact, a diagonal argument yields $\phi\colon\Nat\to\Nat$ increasing such that $(x_{\phi(j)}(n))_{j=1}^\infty$ converges for every $n\in\Nat$. Then, if
\[
y_j=x_{\phi(2j-1)}-x_{\phi(2j)},
\]
$\lim_j y_j(n)=0$ for all $n\in\Nat$. By subsymmetry, $\YY=(y_j)_{j=1}^\infty$ is equivalent to $\XX$ and so it is subsymmetric. In particular, $\YY$ is semi-normalized.

By the gliding-hump technique and the principle of small perturbations (see, e.g., the proof of \cite{AlbiacKalton2016}*{Theorem 1.3.10}), there is $\psi\colon\Nat\to\Nat$ such that $(y_{\psi(j)})_{j=1}^\infty$ is equivalent to a sequence $\ZZ=(z_j)_{j=1}^\infty$ with $(\supp(z_j))_{j=1}^\infty$ pairwise disjoint. As before, we infer that $\ZZ$ is semi-normalized. Now, it is easy to see that $\ZZ$ is equivalent to the unit vector system of $\ell_p$. By subsymmetry so is $\XX$.
\end{proof}

The following proposition applies in particular to the spaces $\F_p(\Nat^d)$ and $\F_p(\Int^d)$.

\begin{Proposition}\label{prop:FpSubsym}Let $\MM$ be a metric space. Suppose that there is a point $x\in\MM$ such that every annulus centered at $x$ contains only finitely many points. Then every subsymmetric basic sequence in $\F_p(\MM)$, $0<p\le 1$, is equivalent to the unit vector system of $\ell_p$. Consequently, $\F_p(\MM)$ admits a subsymmetric basis if and only if $\F_p(\MM)\simeq \ell_p$.
\end{Proposition}
\begin{proof}Just combine Proposition~\ref{prop:finiteCoronas} with Lemma~\ref{lem:sy}.
\end{proof}

\subsection{Applications to Banach spaces}\label{subsec:banach}
Recall that a metric space $(\MM,d)$ is self-similar with constant $R>1$ if there is a bijection $f\colon\MM\to \MM$ with
\begin{equation}\label{eq:Rclosed}
d(f(x),f(y)) = R d(x,y), \quad x,y\in\MM.
\end{equation}
By Banach's contraction principle, if $\MM$ is complete there is always a fixed point of $f$, and such a fixed point is unique. So, when dealing with self-similar pointed metric spaces we will always assume that the base point $0$ satisfies $f(0)=0$ for the bijection $f$ fulfilling \eqref{eq:Rclosed}.

Obvious examples of self-similar spaces are, of course, quasi-Banach spaces and their subsets closed under multiplication by $R$ and $1/R$ for some $R\in\Rea\setminus\{0,\pm 1\}$. Other examples are Carnot groups (see e.g. \cite{LeD}).

Given a pointed metric space $(\MM,d,0)$ and $c>0$, $B_\MM(c)$ will denote the ball centered at the base point $0$ of radius $c$, i.e.,
\[
B_\MM(c)=\{x\in \MM \colon d(0,x)\le c\}.\] The unit ball $B_\MM(1)$ will be denoted by $B_\MM$.

\begin{Theorem}\label{thm:sumsIsoSelfSimilar}Let $(\MM,d,0)$ be a pointed self-similar metric space with constant $R>1$. Let $0<c_1<c_2<\infty$ with $c_2/c_1>R$. Suppose that there is $\mu>1$ such that
$\MM_{(c_1,c_2]}$ is complementably $p$-amenable in $\MM_{(\mu^{-1}c_1,\mu c_2]}$ for some $0<p\le 1$. Then
\[
\F_p(\MM)\simeq \ell_p\left(\F_p(\MM_{(c_1,c_2]})\right).
\]
Moreover, if there is $c_1<c_3<c_2$ such that $\F_p(\MM_{(c_1,c_3]})$ (respectively
$\F_p(\MM_{(c_3,c_2]})$ is complemented in $\F_p(\MM_{(c_1,c_2]})$, then
\[
\F_p(\MM) \simeq \F_p(\MM\setminus B_\MM(c_1)) \quad \text{(resp.\,} \F_p(\MM)\simeq \F_p(B_\MM(c_2)\text{ )}.
\]
\end{Theorem}
\begin{proof}
Since $\MM$ is self-similar with constant $R$, we easily obtain that $\MM_{(R^n,tR^n]}$ is $1$-Lipschitz isomorphic to $\MM_{(1,t]}$ for every $t>0$ and every $n\in\Int$.
Thus, all the spaces $\F_p(\MM_{(c_1R^{n}, c_2R^{n}]})$ are isometric to $\F_p(\MM_{(c_1,c_2]})$ for $n\in\Int$.

We also infer from self-similarity and our assumption that there is a constant $K$ such that $\MM_{(c_1R^{n},c_2R^{n}]}$ is complementably $p$-amenable in $\MM_{(\mu^{-1}c_1R^{n}, \mu c_2R^{n})}$ with constant $K$ for all $n\in\Int$. Hence, applying Theorems~\ref{thm:sumisomorphimGeneralized}, \ref{thm:sumisomorphimGeneralized:1} and \ref{thm:sumisomorphimGeneralized:2} gives the desired results.
\end{proof}

\begin{Remark}
Analogous results to Theorem~\ref{thm:sumsIsoSelfSimilar} hold replacing closed intervals with open, or closed, or left-open and right-closed intervals. We leave the details to the reader.
\end{Remark}

Given a metric space $(\MM,d)$ and $\alpha\in(0,1]$, its snow-flaking $(\MM,d^\alpha)$ is also a metric space. Obviously a function $f\colon\MM\to\NN$ is $C$-Lipschitz when regarded from the metric space $(\MM,d_\MM)$ into the metric space $(\NN,d_\NN)$ if and only if it is $C^\alpha$-Lipschitz when regarded from $(\MM,d_\MM^\alpha)$ into $(\NN,d_\NN^\alpha)$.

Given a Banach space $(X,\|\cdot\|)$, the assumptions of Theorem~\ref{thm:sumsIsoSelfSimilar} are satisfied for the metric space $(X,\|\cdot\|^\alpha)$, where $(X,\|\cdot\|^\alpha)$ is the snow-flaking of $(X,\|\cdot\|)$. This will be generalized in the following two slightly more general results, Lemmas~\ref{lem:gen2} and~\ref{lem:gen1}. Those will be applied below only to Banach spaces and their subsets, in which case the mapping $\sigma$ from the assumptions is given by $\sigma(x,t):=tx$. We refer to Remark~\ref{rem:mogeGeneralSituations} below where possible applications to more general structures are mentioned.

\begin{Definition}\label{def:geodesic}
Let $(\MM,d,0)$ be a pointed metric space and set $A=[0,1]$ (respectively $A=\{0\}\cup[1,\infty)$). A map
\[
\sigma\colon\MM\times A \to \MM
\]is said to be a \emph{self-similar contraction} (resp. \emph{dilation}) if it satisfies the following conditions:
\begin{enumerate}[label={(G.\arabic*)}]
\item\label{it:basic2} $\forall x\in \MM:\quad \sigma_x(0)=0$ and $\sigma_x(1)=x$,
\item\label{it:geodesic2} $\forall x\in \MM\; \forall t,s\in A:\quad d(\sigma_x(t),\sigma_x(s))\leq |s-t|d(x,0)$,
\item\label{it:selfsimilar2} $\forall x,y\in \MM\; \forall t\in A:\quad d(\sigma_x(t),\sigma_y(t))\leq td(x,y)$,
\end{enumerate}
where for every $x\in\MM$ we denote by $\sigma_x$ the mapping $t\mapsto \sigma(x,t)$.
\end{Definition}

Recall that a metric space $(\MM,d,0)$ is said to be \emph{geodesic} if for all $x$, $y\in\MM$ there is $\sigma\colon[0,1]\to\MM$ such that
$\sigma(0)=x$, $\sigma(1)=y$ and $d(\sigma(t),\sigma(s))=d(x,y)|t-s|$ for all $s$, $t\in[0,1]$.

\begin{Lemma}\label{lem:gen2} Let $(\MM,d,0)$ be a geodesic pointed metric space with a self-similar contraction $\sigma$. Let $S>0$ and $\NN\subset\MM$ be such that for every $x\in \NN\setminus B_\MM(S)$ we have $\sigma_x(\tfrac{S}{d(x,0)})\in \NN$.
Then $\NN_{[0,S]}$ is a $2$-Lipschitz retract of $\NN$. More precisely, the $2$-Lipschitz retraction $r\colon\NN\to \NN_{[0,S]}$ is given by $r(0)=0$ and
\begin{equation}\label{eq:retr2}
r(x):=\sigma_x\left(\min\left\{1,\tfrac{S}{d(x,0)}\right\}\right),\qquad x\in\NN\setminus\{0\}.
\end{equation}
\end{Lemma}
\begin{proof}
It suffices to show that the mapping $r$ given by \eqref{eq:retr2} is $2$-Lipschitz. Pick $x,y\in\NN$ and assume without loss of generality that $d(x,0)\leq d(y,0)$. If $d(y,0)\leq S$ then obviously $d(r(x),r(y)) = d(x,y)$. If $d(x,0)\geq S$ then we have
\begin{align*}
d(r(x),r(y))
& \leq d\left(\sigma_x\big(\tfrac{S}{d(x,0)}\big), \sigma_y\big(\tfrac{S}{d(x,0)}\big)\right) + d\left(\sigma_y\big(\tfrac{S}{d(x,0)}\big), \sigma_y\big(\tfrac{S}{d(y,0)}\big)\right)\\
&\leq \frac{S}{d(x,0)}\big(d(x,y) + |d(y,0) - d(x,0)|\big)\\&\leq 2d(x,y),
\end{align*}
where in the second inequality we used the conditions \ref{it:basic2}, \ref{it:geodesic2} and \ref{it:selfsimilar2} from Definition~\ref{def:geodesic}. Finally, if $d(x,0)<S<d(y,0)$, since $\MM$ is geodesic we find $z\in \MM$ such that $d(z,0)=S$ and $d(x,y) = d(x,z)+ d(z,y)$, which implies that
\begin{align*}
d(r(x),r(y))&\leq d(r(x),r(z)) + d(r(z),r(y))\\
&\leq 2(d(x,z) + d(z,y))\\
& = 2d(x,y).
\qedhere
\end{align*}
\end{proof}

\begin{Lemma}\label{lem:gen1} Let $(\MM,d,0)$ be a geodesic pointed metric space with a self-similar dilation $\sigma$. Let $S>0$ and $\NN\subset\MM$ be such that $0\notin\NN$ and for every $x\in \NN\cap B_\MM(S)$ we have $\sigma_x(\tfrac{S}{d(x,0)})\in \NN$. Then, for every $\alpha\in(0,1]$ and each $0<p\le 1$, $(\NN_{[S,\infty)},d^\alpha)$ is complementably $p$-amenable in $(\NN,d^\alpha)$ with constant $3^{1/p}$.
\end{Lemma}
\begin{proof}
Consider the mapping $r\colon(\NN,d^\alpha)\to \F_p(\NN_{[S,\infty)},d^\alpha)$ defined by
\[
r(x):=\begin{cases}\delta(x) & \quad d(x,0)\geq S\\
{\displaystyle \frac{d^\alpha(x,0)}{S^{\alpha}}\delta\left(\sigma_x\left(\tfrac{S}{d(x,0)}\right)\right)} & \quad d(x,0)\leq S.
\end{cases}
\]
It suffices to show that $r$ is $3^{1/p}$-Lipschitz. Pick $x,y\in\NN$ and assume without loss of generality that $d(x,0)\leq d(y,0)$.
If $d(x,0)\geq S$ then obviously $\|r(x)-r(y)\| = d^\alpha(x,y)$.

If $d(y,0)\leq S$ we have
\[
\|r(x)-r(y)\|^p\le A_1+A_2+A_3,
\] where
\begin{align*}
A_1&=\frac{d^{\alpha p}(x,0)}{S^{\alpha p}}\left\Vert\delta\left(\sigma_x\left(\tfrac{S}{d(x,0)}\right)\right) - \delta\left(\sigma_y\left(\tfrac{S}{d(x,0)}\right)\right)\right\Vert^p\\
A_2&=\frac{d^{\alpha p}(x,0)}{S^{\alpha p}}\left\Vert \delta\left(\sigma_y\left(\tfrac{S}{d(x,0)}\right)\right) - \delta\left(\sigma_y\left(\tfrac{S}{d(y,0)}\right)\right)\right\Vert^p,\text{ and}\\
A_3&=\left\Vert\left(\frac{d^{\alpha}(x,0)}{S^{\alpha}}-\frac{d^{\alpha}(y,0)}{S^{\alpha}}\right)\delta\left(\sigma_y\left(\tfrac{S}{d(y,0)}\right)\right)\right\Vert^p.
\end{align*}
If $d(x,0)<S<d(y,0)$ we have
\[\|r(x)-r(y)\|^p\le A_1+A_4+A_5,\]
where
\begin{align*}
A_4&= \frac{d^{\alpha p}(x,0)}{S^{\alpha p}}\left\Vert\delta\left(\sigma_y\left(\tfrac{S}{d(x,0)}\right)\right) - \delta(y)\right\Vert^p \text{ and}\\
A_5&=\left\Vert\left(\frac{d^{\alpha}(x,0)}{S^{\alpha}}-1\right)\delta(y)\right\Vert^p.
\end{align*}

Applying \ref{it:basic2}, \ref{it:geodesic2}, \ref{it:selfsimilar2} in Definition~\ref{def:geodesic} yields
\begin{align*}
A_1&\le d^{\alpha p}(x,y),\\
\max\{A_2^{1/p}, A_3^{1/p} \} &\le d^{\alpha }(y,0) - d^{\alpha }(x,0), \text{ and }\\
\max\{A_4^{1/p}, A_5^{1/p}\}
&\le B:= \frac{S^{\alpha} - d^{\alpha}(x,0)}{S^{\alpha}}d^{\alpha}(y,0).
\end{align*}
Moreover, if $S<d(y,0)$,
\[
B=d^{\alpha}(y,0) - \frac{d^{\alpha}(x,0) d^{\alpha}(y,0)}{S^\alpha}\le d^{\alpha }(y,0) - d^{\alpha }(x,0).
\]

Summing up, whenever $d(x,0) < S$ we have
\[
\|r(x)-r(y)\|^p \le d^{\alpha p }(x,y) + 2 \left(d^{\alpha }(y,0) - d^{\alpha }(x,0)\right)^p\le 3 d^{\alpha p }(x,y).\qedhere
\]
\end{proof}
\begin{Remark}\label{rem:mogeGeneralSituations}Even though we will apply Lemmas~\ref{lem:gen2} and ~\ref{lem:gen1} only to Banach spaces and their subsets, it is worth it mentioning that there are more general situations to which they might be applied. The assumptions of Lemma~\ref{lem:gen2} are satisfied for metric spaces with \emph{conical geodesic bicombing} (such convex sets in Banach spaces, $CAT(0)$ spaces, or Busemann spaces), see \cite{DL15}. Actually the assumptions of Lemma~\ref{lem:gen2} are motivated by the notion of conical geodesic bicombing.

Moreover, there is also a class of self-similar spaces satisfying both Lemma~\ref{lem:gen2} and Lemma~\ref{lem:gen1} which extend their applicability beyond the class of Banach spaces. Those are \emph{first layers $V_1(G)$ of metric scalable groups $G$} (in particular, first layers of Carnot groups); see \cite{DLM20} for the corresponding definitions. The fact that any $V_1(G)$ as above satisfies Lemma~\ref{lem:gen2} and Lemma~\ref{lem:gen1} easily follows from \cite{DLM20}*{Lemma 3.1}. Thus, Lipschitz-free $p$-spaces over $V_1(G)$ are isomorphic to their $\ell_p$-sum for every $p\in(0,1]$.
\end{Remark}

Lemma~\ref{lem:gen2}, Lemma~\ref{lem:gen1} and Lemma~\ref{lem:separatedpoint} easily yield the following.

\begin{Corollary}\label{cor:generalized}Let $X$ be a Banach space and let $\alpha, p\in (0,1]$.
\begin{enumerate}[label={(\roman*)},leftmargin=*,widest=iii]
\item If $\MM\subset X$ has non-empty interior, then
\[
\F_p(B_X,\Vert \cdot\Vert^\alpha)\unlhd_{2^\alpha} \F_p(\MM,\Vert \cdot\Vert^\alpha).
\]
\item If $\MM\subset X$ is a bounded set, then
\[
\F_p(X\setminus B_X,\Vert \cdot\Vert^\alpha)\unlhd_{3^{1/p}} \F_p(X\setminus \MM,\Vert \cdot\Vert^\alpha).
\]
\item For every $0<s<S<\infty$, $X_{[s,S]}$ is complementably $p$-amenable in $X\setminus\{0\}$ with constant $2\cdot 3^{1/p}$.
\item Moreover, there is a constant $C = C(p,\alpha)>0$ (depending on $p,\alpha$ but not on $X$) such that if $0<s<S<\infty$ and $\MM\subset X$ is closed under multiplication by $\frac{S}{d(x,0)}$ for every $x\in\MM\setminus B(0,S)$, closed under multiplication by $\frac{s}{d(x,0)}$ for every $x\in\MM\cap B(0,s)$, and we consider $\MM$ endowed with the snow-flaking $\Vert \cdot\Vert^\alpha$, then for every $0<s<S<\infty$, $\MM_{[s,S]}$ is complementably $p$-amenable in $\MM$ with constant $C$.
\end{enumerate}
\end{Corollary}

Our following theorem improves two results of P. Kaufmann from \cite{K15}, namely \cite{K15}*{Theorem 3.1 and Corollary 3.3}.
\begin{Theorem}\label{thm:Banach} Let $X$ be a Banach space, $0<p\le 1$ and $0<\alpha\leq 1$. Then for any $R>1$ and any subset $\MM$ of $X$ closed under multiplication by nonnegative numbers (that is, $\bigcup_{\lambda\geq 0}\lambda\MM\subset\MM$), if we consider $\MM$ endowed with the snow-flaking $\Vert \cdot\Vert^\alpha$, we have
\[
\F_p(\MM)\simeq \F_p(B_\MM) \simeq \F_p(\MM\setminus B_\MM) \simeq \ell_p(\F_p\left(\MM_{(1,R]}\right)).
\]
\end{Theorem}
\begin{proof}
Let $0<c_1<c_2<\infty$ and $\mu>1$. By Corollary~\ref{cor:generalized},
$\MM_{[c_1,c_2]}$ is complementably $p$-amenable in $\MM_{(\mu^{-1}c_1,\mu c_2]}$. Since $\MM_{(c_1,c_2]}$ is dense in $\MM_{[c_1,c_2]}$, $\MM_{(c_1,c_2]}$ also is complementably $p$-amenable in $\MM_{(\mu^{-1}c_1,\mu c_2]}$. Then, the result follows from Theorem~\ref{thm:sumsIsoSelfSimilar}.
\end{proof}

\begin{Remark}Note that since $\lim_{t\to 1} t x=x$ for every $x\in X$, the type of intervals we deal with in Theorem~\ref{thm:Banach} is irrelevant.
\end{Remark}

\begin{Corollary}
Let $X$ be a Banach space, $0<\alpha\le 1$ and $0<p\leq 1$. Then
\[
\F_p(X,\Vert \cdot\Vert^\alpha)\simeq \ell_p(\F_p(S_X\oplus [0,1],\Vert \cdot\Vert^\alpha)).
\]
\end{Corollary}
\begin{proof}
By Theorem~\ref{thm:Banach} we have
$\F_p(X,\Vert \cdot\Vert^\alpha)\simeq \ell_p(\F_p(X_{[1,2]},\Vert \cdot\Vert^\alpha ))$. Moreover, it is easy to see that the ``polar'' map
\[
x\mapsto \left(\frac{x}{\|x\|}, \|x\| \right)
\]
defines a Lipschitz isomorphism from $X_{[1,2]}$ onto $S_X\oplus [1,2]$. Since the intervals $[1,2]$ and $[0,1]$ are isometric we are done.
\end{proof}

\begin{Corollary}\label{cor:NEI}Let $X$ be a Banach space. Suppose $\MM$ is a subset of $X$ with nonempty interior. Then for $0<\alpha\le 1$ and $0<p\le 1$ we have $\F_p(X,\Vert \cdot\Vert^\alpha)\unlhd \F_p(\MM,\Vert \cdot\Vert^\alpha)$.
\end{Corollary}

\begin{proof}Just combine Corollary~\ref{cor:generalized} with Theorem~\ref{thm:Banach}.
\end{proof}

\begin{Corollary}\label{cor:euclidean}
Let $d\in\Nat$ and $0<p\le 1$. Then $\F_p(\Rea^d)\simeq \F_p(\Rea_+^d)$.
\end{Corollary}
\begin{proof} The result follows from Theorem~\ref{thm:Banach} in combination with the fact that $[-1,1]^d$ and $[0,1]^d$ are Lipschitz isomorphic.
\end{proof}

The following result will be further improved in Corollary~\ref{cor:doublingSnowflake}.

\begin{Corollary}\label{cor:sfEuclidean}Let $X$ be a finite-dimensional Banach space. Then for $0<\alpha<1$, $\F(X,\Vert\cdot\Vert^\alpha)\simeq\ell_1$.
\end{Corollary}
\begin{proof}
It follows from Theorem~\ref{thm:Banach} and the fact that $\F(K,\Vert\cdot\Vert^\alpha)\simeq \ell_1$ whenever $K\subset X$ is an infinite compact set (see \cite{WeaverBook2018}*{Theorems 4.38 and 8.49}
).
\end{proof}

Given $d\in\Nat$, $S^{d-1}$ and $B^d$ denote, respectively, the Euclidean sphere and the Euclidean ball of $\Rea^d$.

\begin{Theorem}\label{thm:sphere}For $d\in\Nat$, $0<p\le 1$ and $0<\alpha\le 1$, we have
\[
\F_p(S^d, |\cdot|^\alpha)\simeq \F_p(\Rea^d, |\cdot|^\alpha).
\]
\end{Theorem}
\begin{proof}
Consider $S^d$ equipped with the Euclidean distance and choose the ``north'' $\nu=(0,\dots,0,1)$ as base point of $S^d$. If we denote
\[
S^d[s,t]=\{x=(x_i)_{i=1}^{d+1} \in S^d \colon s\le x_{d+1} \le t\}, \quad -1\le s <t\le 1,
\]
and define
\[
\eta(s)=\max\left\{1-\frac{s^2}{2},-1\right\}, 0<s\le\infty,
\]
we have $(S^d)_{[s,t]}=S^d[\eta(t),\eta(s)]$ for all $0<s<t<\infty$. Hence, applying Theorem~\ref{thm:complemented} and Lemma~\ref{lem:separatedpoint} with
\[
(I_n)_{n\in N}=\{ (-\infty,1/2), (0,\infty)\}
\]
and $R=2^{\alpha}$ yields
\[
\F_p(S^d, |\cdot|^\alpha) \unlhd \F_p(S^d[0,1],|\cdot|^\alpha)\oplus\Rea\oplus \F_p(S^d[-1,1/2],|\cdot|^\alpha).
\]
The stereographic projection $M_\nu$ from the north point, given by
\[
(x_i)_{i=1}^{d+1} \mapsto \left( \frac{x_i}{1-x_{d+1}} \right)_{i=1}^d,
\]
is a diffeomorphism from $S^d\setminus\{\nu\}$ onto $\Rea^d$. Moreover, for every $-1\le h \le 1$,
\[
M_\nu ( \{ x\in S^d \colon x_{n+1}=h\} )= \left\{ y\in\Rea^d \colon \Vert y \Vert=\xi(h) \right\},
\]
where $\xi(h)=\sqrt{(1+h)/(1-h)}$.
Consequently, $M_\nu$ is a Lipschitz isomorphism from $S^d[-1,1/2]$ onto $\sqrt 3 B^d$ which maps $S^d[-1,0]$ onto $B_d$. We infer that $S^d[-1,1/2]$ and $S^d[-1,0]$ are Lipschitz isomorphic to $B^d$ and that $S^d[-1,0]$ is a Lipschitz retract of $S^d[-1,1/2]$. Therefore $S^d[-1,0]$ is complementably $p$-amenable in $S^d[-1,1/2]$. Applying Theorem~\ref{thm:reversecomplementedGerenalization} with the singleton $(I_n)_{n\in N}=\{(1/2,\infty)\}$ and $R=2^{\alpha}$, and taking into account Lemma~\ref{lem:separatedpoint} we obtain
\[
\F_p(S^d[-1,0],|\cdot|^\alpha)\unlhd \F_p(S^d,|\cdot|^\alpha).
\]
Since $S^d[-1,0]$ and $S^d[0,1]$ are isometric,
\[
\F_p(S^d[-1,0],|\cdot|^\alpha)\simeq \F_p(S^d[0,1],|\cdot|^\alpha).
\]
Combining and applying Theorem~\ref{thm:Banach} together with Pe\l czy\'nski's decomposition method yields $\F_p(S^d,|\cdot|^\alpha)\simeq\F_p(B^d,|\cdot|^\alpha)$. Invoking Theorem~\ref{thm:Banach} completes the proof.
\end{proof}

\section{Applications to doubling metric spaces}\label{sec:doubling}\noindent
In this section we first provide a proof of the fact that for $p\in(0,1]$, every doubling metric space $\NN$ is complementably $p$-amenable in any metric space containing it , which answers in the positive Question 6.7 from \cite{AACD2019}. This is known for $p=1$, but our proof seems to be interesting even for this case because usually the authors refer to several several deep results from \cite{LN05}, while here we give a brief self-contained argument. Further, we collect applications of this fact together with methods developed in preceding sections.

\subsection{Doubling metric spaces are complementably $p$-amenable}
Let us recall that a metric space $\MM$ is \emph{doubling} if there exists a constant $D(\MM)\in\Nat$, called the \emph{doubling constant} of $\MM$, such that every ball of radius $r>0$ in $\MM$ can be covered by at most $D(\MM)$-many balls of radius $r/2$. It is not very difficult to see that every subspace of a doubling metric space $\MM$ is again doubling with doubling constant bounded by $D(\MM)^2$. Euclidean spaces are typical examples of doubling spaces. Doubling metric spaces are precisely the metric spaces of finite Assouad dimension. The purpose of this subsection is to prove the following result.

\begin{Theorem}\label{thm:doubling}Let $(\MM,d)$ be a metric space and $\NN$ be a closed subset of $\MM$ with finite doubling constant $D\geq 2$. For each $0<p\le 1$, $\NN$ is complementably $p$-amenable in $\MM$ with constant at most $C(p) \, D^{4/p}$.
To be precise, $C(p)=112\cdot 15^{1/p}$.
\end{Theorem}

The proof for $p=1$ was given in \cite{LP13}, where the authors observed that it follows from deep results of Lee and Naor \cite{LN05}. Recently, an easier proof of the essential ingredient by Lee and Naor was given in \cite{BMS18} and this approach actually admits generalization to the case $p<1$, which is what we indicate in this subsection. Moreover, we present a self-contained and easier argument even for the case $p=1$ at the cost of getting a worse estimate ($D^4$ instead of $\log D$ which is the estimate for $p=1$ from \cite{LP13}). Let us give some more details.

First we need the following preliminary result which is more-or-less the content of \cite{BMS18}*{Lemma 2.4}. For the convenience of the reader we include the proof here.
\begin{Proposition}\label{prop:partition}
Let $(\MM,d)$ be a metric space and $\NN\subset\MM$ be a closed subset with finite doubling constant $D\geq 2$. Then there exists a countable
family $(V_i,\phi_i, x_i)_{i\in I}$ such that:
\begin{enumerate}[label={(H.\arabic*)}]
\item\label{it:partition1} $x_i\in\NN$ and $d(x_i,x)\leq 7 d(x,\NN)$ for all $i\in I$ and all $x\in V_i$;
\item\label{it:partition3} $(V_i)_{i\in I}$ is a $3D^4$-overlapping open cover of $\MM\setminus\NN$;
\item\label{it:partition2} for each $i\in I$ the mapping $\phi_i\colon\MM\setminus\NN\to [0,1]$ is $1$-Lipschitz with $\{x\in\MM\setminus\NN\colon \phi_i(x)>0\}\subset V_i$; and
\item\label{it:partition4} for every $x\in \MM\setminus\NN$ there exists $i\in I$ with $\phi_{i}(x) > d(x,\NN)/4$.
\end{enumerate}
\end{Proposition}
\begin{proof} For $n\in\Int$, let $\NN_n$ be a maximal $2^n$-separated subset of $\NN$, i.e., $d(y,z)\ge 2^n$ for all $y$, $z\in\NN_n$ with $y\not=z$ and $d(x,\NN_n)<2^n$ for all $x\in \NN$. Since $\NN_n$ intersects finitely many elements of any ball of $\MM$, $\inf_{y\in\NN_n} d(x,y)$ is attained for all $x\in\MM$, so that the annulus
\[
W_{n}=\{x\in\MM\setminus \NN \colon 2^n\leq d(x,\NN)<2^{n+1} \}
\]
is covered by the family
\[
W_{(y,n)}=\{x\in W_n \colon d(x,y)\le d(x,z) \text{ for all } z \in \NN_n\}, \quad y\in \NN_n.
\]
In turn, $(W_n)_{n\in\Int}$ is a partition of $\MM\setminus\NN$. Therefore, if we put $I=\{(y,n)\colon n\in \Int, y\in \NN_n\}$, $(W_i)_{i\in I}$ is a cover of $\MM\setminus\NN$.
For each $(y,n)\in I$, put $x_{(y,n)}=y$ and
\[
V_{(y,n)}=\{x\in\MM\setminus\NN\colon d(x,W_{(y,n)}) <2^{n-1}\}.
\]
Since $W_i\subseteq V_i$ for all $i\in I$, $(V_i)_{i\in I}$ is an open cover of $\MM\setminus\NN$. Define
\[
\phi_i(x):=d(x,\MM\setminus V_i), \quad x\in\MM\setminus\NN,\,\in I,
\]
so that \ref{it:partition2} trivially holds. We claim that $(V_i,\phi_i, x_i)_{i\in I}$ is the desired family. We start by proving that if $(y,n)\in I$ and $x\in V_{(y,n)}$ then
\begin{align}
2^{n-1}&\le d(x,\NN)< 5 \cdot 2^{n-1}, \text{ and }\label{eq:part1} \\
d(y, x)&\le 7 \cdot 2^{n-1},\label{eq:part2}
\end{align}
which easily yields \ref{it:partition1}.
Indeed, there is $x'\in W_{(y,n)}$ with $d(x,x')<2^{n-1}$. Since $x'\in W_n$ there is $z\in\NN$ with $d(x',z)<2^{n+1}$. The properties of $\NN_n$ yield $y'\in\NN_n$ with $d(z,y')<2^n$. Since $d(x',y)\leq d(x',y')$,
\begin{align*}
d(x,\NN)&\le d(x',\NN)+d(x,x')< 2^{n+1}+2^{n-1}=5\cdot 2^{n-1},\\
d(x,\NN)&\ge d(x',\NN)-d(x,x')\ge 2^n-2^{n-1}=2^{n-1}, \text{ and} \\
d(x,y)&\leq d(x,x')+d(x',y)\leq d(x,x')+d(x',y')\\
&\leq d(x,x')+d(x',z)+d(z,y')\\
&\leq 2^{n-1}+ 2^{n+1}+ 2^{n}=7\cdot 2^{n-1}.
\end{align*}

\ref{it:partition3}: Given $x\in\MM\setminus\NN$, put $K_n=\{y\in\NN_n \colon x\in V_{(y,n)}\}$ for each $n\in\Int$. Suppose that $K_n\not=\emptyset$ and pick $y\in K_n$. If $z\in K_n$ inequality \eqref{eq:part2} yields
\[
d(y,z) \leq d(y,x)+d(x,z)\le 14\cdot 2^{n-1} = 7 \cdot 2^n.
\]
Therefore $K_n$ is contained in the ball $B(y,7 \cdot 2^n)$. In turn, by the doubling property, $B(y,7 \cdot 2^n)$ is covered by $D^4$ balls of radius $7\cdot 2^{n-4}$. If one of these balls contains two different points $z_1$, $z_2\in K_n$ we reach the absurdity $d(z_1,z_2)\le 14\cdot 2^{n-4}<2^n$, so $|K_n|\le D^4$. Moreover, if $j\in\Int$ is such that $d(x,\NN)\in [2^j, 2^{j+1})$, inequality \eqref{eq:part1} yields $n\in\{j-1,j,j+1\}$.
We infer that $|\{ i\in I \colon x\in V_i\}|\le 3 D^4$.

\ref{it:partition4}: Let $x\in\MM\setminus\NN$ and pick $i=(y,n)\in I$ with $x\in W_i$. If $z\in\MM\setminus V_i$ by definition we have $d(z,x)\ge d(z,W_i)\ge 2^{n-1}$, so $\phi_i(x)\ge 2^{n-1}$. Since, by definition, $d(x,\NN)<4\cdot 2^{n-1}$, we are done.
\end{proof}

\begin{proof}[Proof of Theorem~\ref{thm:doubling}]
Let $ \VV=(V_i,\phi_i, x_i)_{i\in I}$ be as in Proposition~\ref{prop:partition}, and set $K=3D^4$, so that $(V_i)_{i\in I}$ is $K$-overlapping. Hence, by property \ref{it:partition3}, of $\VV$, we candefine $\Phi\colon \MM\setminus\NN\to [0,\infty)$ by $\Phi=\sum_{i\in I}\phi_i$. Moreover, since for each $x$, $y\in\MM\setminus \NN$ the cardinality of the set
\[
I_{x,y}=\{i\in I \colon \phi_i(x)\not=0 \text{ or } \phi_i(y)\not=0\}
\]
is at most $2K$, $\Phi$ is $(2K)$-Lipschitz. Besides, by property~\ref{it:partition4} of $\VV$, $\Phi(x)\ge d(x,\NN)/4$ for all $x\in\MM\setminus\NN$. Hence, for each $i\in I$ we can define $\psi_i\colon\MM\setminus\NN\to [0,1]$ by $\psi_i=\phi_i/\Phi$. Of course, $\sum_{i\in I} \psi_i=1$. Consider $f\colon\MM\to \F_p(\NN)$ given by
\[
f(x):=\begin{cases}
\delta(x) & x\in\NN,\\
\sum_{i\in I} \psi_i(x)\delta(x_i) & x\in\MM\setminus\NN.
\end{cases}
\]
By Lemma~\ref{lem:complementedCondition}, it suffices to show that $f$ is $C$-Lipschitz for $C= 112 (5K)^{1/p}$. First, we prove that for all $x$, $y\in\MM\setminus\NN$
\begin{equation}\label{eq:crucial}
\sum_{i\in I}|\psi_i(x)-\psi_i(y)|^p\leq \frac{2 \cdot 8^p\cdot K}{A_{x,y}^p}d^p(x,y),
\end{equation}
where $A_{x,y}= \max\{d(y,\NN),d(x,\NN)\}$. We may assume without loss of generality that $d(x,\NN)\leq d(y,\NN)$. Using that the functions $\phi_i$ are $1$-Lipschitz
we obtain
\begin{align*}
|\psi_i(x)-\psi_i(y)| &=\frac{|\phi_i(x)(\Phi(y)-\Phi(x))+\Phi(x) (\phi_i(x)-\phi_i(y))|}{\Phi(x)\Phi(y)}\\
&\le \frac{\phi_i(x)|\Phi(y)-\Phi(x)|+\Phi(x) |\phi_i(x)-\phi_i(y)|}{\Phi(x)\Phi(y)}\\
&\le \frac{2K \phi_i(x)d(x,y)+\Phi(x) d(x,y)}{\Phi(x)\Phi(y)}.
\end{align*}
Thus, summing over $I_{x,y}$, we get
\[
\sum_{i\in I_{x,y}} |\psi_i(x)-\psi_i(y)| \leq \frac{4K}{\Phi(y)}d(x,y)\leq \frac{16K}{d(y,\NN)}d(x,y)=\frac{16K}{A_{x,y}}d(x,y).
\]
Applying H\"older's inequality yields \eqref{eq:crucial}. We also infer from H\"older's inequality that for all $x\in\MM\setminus\NN$
\begin{equation}\label{eq:pnorm}
\sum_{i\in I} |\psi_i(x)|^p \le K^{1-p} \left(\sum_{i\in I} |\psi_i(x)|\right)^{1/p}\le K^{1-p}.
\end{equation}

Let $x$, $y\in\MM$. If $\{x,y\}\subset\NN$ then of course we have $\|f(x)-f(y)\| = d(x,y)$. Suppose that $x\in\NN$ and $y\in \MM\setminus\NN$. By properties~\ref{it:partition1} and \ref{it:partition2} of $\VV$,
$d(x_i,y)\le 7 d(x,y)$ whenever $\psi_i(y)\neq 0$. Hence, taking also into account \eqref{eq:pnorm},
\begin{align*}
\|f(y)-f(x)\|^p & = \left\Vert\sum_{i\in I} \psi_i(y)(\delta(x_i)-\delta(y) + \delta(y) - \delta(x))\right\Vert^p\\
& \leq \sum_{i\in I} \psi_i^p(y)\left(d^p(x_i,y)+d^p(y,x)\right)\\
& \leq (7^p+1) d^p(x,y) \sum_{i\in I} \psi_i^p(y) \\
& \leq K^{1-p}(7^p+1)d^p(x,y).\\
\end{align*}

It remains to deal with the case $\{x,y\} \subseteq \MM\setminus\NN$. Suppose that there exists $j\in I$ with $\{x,y\} \subset V_j$. If $x\in V_i$ we have
\[
d(x_i,x_j)\le d(x_i,x)+d(x,x_j)\le 2\cdot 7 d(x,\NN).
\]
Since the same holds if $y\in V_i$, we have
$
d(x_i,x_j)\le 14 A_{x,y}
$
for all $i\in I_{x,y}$. Combining this piece of information with inequality \eqref{eq:crucial} we obtain
\begin{align*}
\|f(x)-f(y)\|^p & =\left \|\sum_{i\in I} (\psi_i(x)-\psi_i(y))(\delta(x_i)-\delta(x_{j}))\right\|^p\\
& \leq \sum_{i\in I_{x,y} } |\psi_i(x)-\psi_i(y)|^p d^p(x_i,x_{j})\\
& \leq 14^p A^p_{x,y} \sum_{i\in I}|\psi_i(x)-\psi_i(y)|^p\\
& \le 2 \cdot 112^p K d^p(x,y).
\end{align*}
Finally, assume that $\{i\in I\colon x\in V_i\}\cap\{i\in I\colon y\in V_i\}=\emptyset$. Taking into account \eqref{eq:pnorm} we have
\begin{align*}
\|f(x)-\delta(y)\|^p & \leq \left\|\sum_{i\in I} \psi_i(x)(\delta(x_i)-\delta(y))\right\|^p\\
&\le \sum_{i\in I} \psi_i^p(x) d^p(x_i,y)\\
&\le \sum_{i\in I} \psi_i^p(x) (d^p(x_i,x)+ d^p(x,y))\\
&\le (7^p d^p(x,\NN)+ d^p(x,y) ) \sum_{i\in I} \psi_i^p(x)\\
&\le K^{1-p} d^p(x,y) + 7^p A^p_{x,y} \sum_{i\in I} \psi_i^p(x).
\end{align*}
Considering also the inequality that we obtain from switching the roles of $x$ and $y$, and using again \eqref{eq:crucial} yields
\begin{align*}
\|f(x)-f(y)\|^p
&\le \|f(x)-\delta(y)\|^p + \|\delta(y)-\delta(x)\|^p+\|\delta(x)-f(y)\|^p\\
&\le (2K^{1-p} + 1) d^p(x,y) +7 A^p_{x,y} \left(\sum_{i\in I} \psi_i^p(x) + \psi_i^p(y)\right)\\
&= (2K^{1-p} + 1) d^p(x,y)+7^p A^p_{x,y}\sum_{i\in I} |\psi_i(x) - \psi_i(y)|^p\\
&\le (2K^{1-p} + 1+ 2 \cdot 56^p\cdot K) d^p(x,y).
\end{align*}
Combining the inequalities and comparing the constants in the estimates yields
$\|f(x)-f(y)\| \le C d(x,y)$.
\end{proof}

\subsection{Applications}

The following result answers in the positive \cite{AACD2019}*{Question 6.7}.

\begin{Corollary}\label{cor:doubling}
Let $\MM$ be a doubling metric space and $\NN\subset \MM$. Then for $p\in(0,1]$, $\NN$ is complementably $p$-amenable in $\MM$ with constant depending only on $p$ and $\MM$. In particular, $\F_p(\MM)$ has the $\pi$-property.
\end{Corollary}

\begin{proof} We can assume without loss of generality that $\NN$ is closed. Since $\NN$ is a doubling metric space with doubling constant depending only on that of $\MM$, the first part of the corollary holds. To prove that $\F_p(\MM)$ has the $\pi$-property, we order the set
\[
I=\{ \NN \subset \MM \colon |\NN|<\infty\}
\]
by inclusion, and for each $\NN\in I$ we choose $T_\NN=L_{\jmath}\circ P_\NN$, where $P_\NN\colon \F_p(\MM)\to \F_p(\NN)$ is the map provided by the complementable $p$-amenability of $\NN$ and $L_\jmath\colon \F_p(\MM)\to \F_p(\NN)$ is the canonical map.
\end{proof}

\begin{Corollary}\label{cor:fddDoubling}
Let $\MM$ be a complete countable doubling metric space. Then $\F(\MM)$ has the FDD property.
\end{Corollary}
\begin{proof}
By Corollary~\ref{cor:doubling}, $\F(\MM)$ has the $\pi$-property and since any complete doubling metric space is proper, by \cite{D15} $\F(\MM)$ has MAP. Since any separable Banach space with $\pi$-property and MAP has the FDD property (see \cite{CasHandbook}*{Theorem 4.6 and Theorem 6.3}), we are done. \end{proof}

We also immediately obtain the following interesting result which applies, e.g., to Carnot groups.
\begin{Corollary}\label{cor:carnotGroup}
Let $p\in(0,1]$ and $(\MM,d)$ be a self-similar doubling metric space. Then there exists $R>1$ such that
\[
\F_p(\MM) \simeq \F_p(B_\MM)\simeq \F_p(\MM\setminus B_\MM)\simeq \ell_p\left(\F_p(\MM_{(1,R]}))\right).
\]
\end{Corollary}
\begin{proof}
Just apply Theorem~\ref{thm:sumsIsoSelfSimilar} and Theorem~\ref{thm:doubling}.
\end{proof}

The following corollary is a generalization of \cite{K15}*{Corollary 3.5}. Recall that a metric space $\MM$ is \emph{Lipschitz homogeneous} if for every $x,y\in\MM$ there is a Lipschitz isomorphism $f\colon\MM\to\MM$ with $f(x)=y$.

\begin{Corollary}
Let $(\MM,d)$ be a pointed self-similar doubling Lipschitz homogeneous metric space. Then given $\NN\subset\MM$ with non-empty interior we have $\F_p(\MM)\simeq\F_p(\NN)$ for all $p\in(0,1]$.
\end{Corollary}
\begin{proof}
Since $\MM$ is Lipschitz homogeneous, all the balls are Lipschitz equivalent. Thus, if $B\subset\NN$ is a ball, using Theorem~\ref{thm:doubling} and Corollary~\ref{cor:carnotGroup} we obtain
\[
\F_p(\NN)\unlhd \F_p(\MM)\simeq \F_p(B_\MM)\simeq\F_p(B)\unlhd \F_p(\NN).
\]
By Corollary~\ref{cor:carnotGroup} we have $\F_p(\MM)\simeq\ell_p(\F_p(\MM))$ and so an application of Pe{\l}czy\'nski's decomposition method finishes the proof.
\end{proof}

We next see a generalization of the essentially known result that $\F(K,d^\alpha)\simeq \ell_1$ whenever $(K,d)$ is an infinite compact set in a doubling metric space (see
\cite{WeaverBook2018}*{Theorems 4.38 and 8.49}). We do not aim here to prove its analogue for $p<1$.

\begin{Corollary}\label{cor:doublingSnowflake}
Let $(\MM,d)$ be a doubling metric space and $0<\alpha<1$. Then $\F(\MM,d^\alpha)\simeq \ell_1$.
\end{Corollary}
\begin{proof}
Pick $0<\alpha<\beta<1$. By Assouad's theorem (see \cite{Assouad1983}), $(\MM,d^{\alpha/\beta})$ admits a bi-Lipschitz embedding in some Euclidean space. Thus, we assume without loss of generality that $(\MM,d^\alpha)$ is a subset of $(X,d^\beta)$ for some finite-dimensional Banach space $X$. Since $(X,d^\beta)$ is a doubling metric space, from Theorem~\ref{thm:doubling} we obtain $\F(\MM,d^\alpha)\unlhd \F(X,d^\beta)$. By Corollary~\ref{cor:sfEuclidean}, $\F(X,d^\beta)\simeq\ell_1$. Since $\ell_1$ is a prime Banach space (see \cite{Pel1960}) we are done.
\end{proof}

Let us consider the case when the doubling metric space in question is not self-similar (a typical example is $\Int^d$).

Given $R\in(0,1)\cup(1,\infty)$, we shall say that a pointed metric space $(\MM,d,0)$ is \emph{$R$-closed} if there exists a map $f\colon \MM\to \MM$ fulfilling \eqref{eq:Rclosed} with $f(0)=0$. Note that, in the case when $R<1$ and $\MM$ is complete, there is always a fix point of a mapping $f$ satisfying \eqref{eq:Rclosed}, so the condition $f(0)=0$ is redundant. Both metric spaces $\Nat^d$ and $\Int^d$ are $2$-closed and doubling.

\begin{Theorem}\label{thm:doublingAndClosed}
Let $(\MM,d,0)$ be a pointed doubling metric space which is $R$-closed for some $R\in(0,1)\cup(1,\infty)$. Let $N$ be a countable set and $\phi\colon N\to\Nat_*$ unbounded. Then:
\begin{enumerate}[label={(\roman*)}, leftmargin=*,widest=ii]

\item\label{it:notBall} If $R>1$, for every $c>R$ we have
\[
\ell_p(\F_p(\MM\setminus B_{\MM}))\simeq \F_p(\MM\setminus B_{\MM}) \simeq \left(\bigoplus_{n\in N} \F_p(\MM_{
(R^{\phi(n)}, c R^{\phi(n)}]
})\right)_p.
\]
\item\label{it:ball} If $R<1$, for every $0<c<R$ we have
\[
\ell_p(\F_p(B_{\MM})) \simeq\F_p(B_{\MM}) \simeq
\left(\bigoplus_{n\in N}
\F_p(\MM_{
( c R^{\phi(n)}, R^{\phi(n)}]
})\right)_p.
\]
\end{enumerate}
In particular, if $\MM$ is uniformly separated and $R>1$, or $\MM$ is bounded and $R<1$, we have $\F_p(\MM)\simeq \ell_p(\F_p(\MM))$.
\end{Theorem}
\begin{proof}
Let us prove \ref{it:notBall}. Pick an arbitrary $\mu>0$. For $n\in\Int$, set $\MM_n:=\MM_{(R^{n}, c R^{n}]}$ and $\NN_n:=\MM_{(\mu^{-1}R^{n}, \mu cR^{n}]}$. Let $f\colon\MM\to \MM$ be a map such that $d(f(x),f(y)) = R d(x,y)$ for all $x$, $y\in\MM$ with $f(0)=0$. Note that $f^k$ is a $1$-Lipschitz isomorphism onto its image for all $k\in\Nat$. Moreover, $f^k(\MM_n )\subseteq\MM_{n+k}$. By Theorem~\ref{thm:doubling}, there is a constant $K_1$ such that $\F_p(\MM_n)$ is $K_1$-complemented in $\F_p(\MM_m)$ for every $n\le m$. Theorem~\ref{thm:doubling} also yields a constant $K_2$ such that $\MM_n$ is complementably $p$-amenable in $\NN_n$ with constant $K_2$ for every $n\in \Int$. Hence, the result follows from Theorem~\ref{thm:sumisomorphimGeneralized}.

The proof of \ref{it:ball} is analogous and so we omit it. If $\MM$ is uniformly separated (resp. bounded) we have $\MM=\MM\setminus B_\MM$ (resp. $\MM=B_\MM$) under a suitable rescaling of the metric. Since rescaling of the metric space gives isometric Lipschitz free spaces, we are done
\end{proof}

\begin{Remark}A theorem analogous to Theorem~\ref{thm:doublingAndClosed} holds for open, closed, left-closed, and right-open intervals.
\end{Remark}

Recall that a \emph{net} in a metric space $\MM$ is an $a$-separated and $b$-dense subset in $\MM$ for some positive numbers $a, b$. A typical example of a net in $\Rea^d$ is the set $\Int^d$. The following corollary extends and improves the result \cite{HN17}*{Theorem 7} of P. H\'ajek and M. Novotn\'y on $\ell_{1}$-sums of Lipschitz free spaces over nets.

\begin{Corollary}\label{cor:netInCarnotGroup}
Let $\MM$ be a pointed doubling self-similar metric space and let $\NN\subset \MM$ be a net in $\MM$. Then, for every $p\in(0,1]$, we have
\[
\F_p(\NN)\simeq \ell_p(\F_p(\NN)).
\]
\end{Corollary}
\begin{proof}
It is not difficult to construct a net $\NN_0\subset \MM$ in $\MM$ which is $R$-closed for some $R>1$ (see the proof of \cite{CDL19}*{Corollary 1.18}). Thus, by Theorem~\ref{thm:doublingAndClosed}, $\F_p(\NN_0)\simeq \ell_p(\F_p(\NN_0))$. If $\NN\subset\MM$ is an arbitrary net in $\MM$, using that $\MM$ is unbounded and separable, hence all the nets are infinite and countable, by \cite{AACD2019}*{Theorem 3.5} we get that $\F_p(\NN)\simeq \F_p(\NN_0)$, and we are done.
\end{proof}

The following is an improvement of \cite{CDL19}*{Corollary 1.18}.

\begin{Corollary}\label{cor:isoLipSpaces}
Let $(\MM,d,0)$ be a pointed doubling self-similar metric space and let $\NN\subset \MM$ be a net in $\MM$. Then $\Lip_0(\NN)\simeq \Lip_0(\MM)$.
\end{Corollary}
\begin{proof}
Just combine \cite{CDL19}*{Proposition 1.9} with Corollaries~\ref{cor:carnotGroup} and \ref{cor:netInCarnotGroup}.
\end{proof}

Corollary~\ref{cor:netInCarnotGroup} gives, in particular, $\F_p(\Int^d)\simeq \ell_p(\F_p(\Int^d))$ for every $d\in\Nat$. We plan on studying in depth the structure of the Lipschitz free $p$-spaces $\F_p(\Int^d)$ in a further publication. For the time being we can state the following.
\begin{Theorem}\label{thm:natAndInt}
For every $p\in(0,1]$ and $d\in\Nat$ we have $\F_p(\Nat^d)\simeq\F_p(\Int^d)$.
\end{Theorem}
\begin{proof}
By Theorem~\ref{thm:doubling}, $\F_p(\Nat^d)\unlhd\F_p(\Int^d)$, and by Theorem~\ref{thm:doublingAndClosed},
\[
\F_p(\Int^d)\simeq \left(\bigoplus_{n\in\Nat} \F_p(\MM_n)\right)_p,
\]
where $\MM_n=\{x\in \Int^d \colon 2^n \le \vert x\vert_\infty\le 2^{n+2}\}$. Moreover, for every $n\in\Nat$ the set $\MM_n$ is isometric to a subset of $\Nat^d$ (use the map $x\mapsto \{2^{n+2}\}^d - x$), so by Theorem~\ref{thm:doubling} the spaces $\F_p(\MM_n)$ are uniformly complemented in $\F_p(\Nat^d)$. Using Corollary~\ref{cor:netInCarnotGroup} we obtain
\[
\F_p(\Int^d)\unlhd\ell_p(\F_p(\Nat^d))\simeq \F_p(\Nat^d).
\]
An application of Pe{\l}czy\'nski's decomposition method completes the proof.
\end{proof}

To close this section let us relate the finite-dimensional structure of the Lipschitz-free space over a self-similar doubling metric space with the Lipschitz-free spaces over its nets. The following two results apply for instance to $\MM=\Rea^d$ and $\NN=\Int^d$.
\begin{Proposition}\label{prop:netAndDoublingSpace}
Let $\MM$ be a doubling self-similar metric space and let $\NN\subset \MM$ be a net in $\MM$. For every $0<p\le 1$ there is a constant $C>0$ and an increasing sequence $(X_n)_{n\in\Nat}$ of subspaces of $\F_p(\MM)$ such that $\overline{\bigcup_{n\in\Nat} X_n} = \F_p(\MM)$ and for all $n\in\Nat$, $X_n$ is $C$-complemented in $\F_p(\MM)$ and $C$-isomorphic to $\F_p(\NN)$.
\end{Proposition}
\begin{proof}
Let $f\colon\MM\to\MM$ be the bijection from the definition of a self-similar space and let $0\in\MM$ be such that $f(0)=0$. It is not difficult to construct a $1$-separated and $1$-dense set $\NN_0\subset \MM$ for which $f(\NN_0)\subset\NN_0$ (see, e.g., the proof of \cite{CDL19}*{Corollary 1.18}). For $n\in\Nat$ put $X_n=[\delta_{\MM}(f^{-n}(x))\colon x\in\NN_0]$. Theorem~\ref{thm:doubling} gives that $X_n$ is uniformly complemented in $\F_p(\MM)$ and uniformly isomorphic to $\F_p(f^{-n}(\NN_0))$, which in turn is isometric to $\F_p(\NN_0)$. Thus, it remains to prove that $\bigcup_{n\in\Nat} X_n$ is dense in $\F_p(\MM)$. For that, it is sufficient to observe that since $\NN_0$ is $1$-dense, $\{f^{-n}(x)\colon n\in\Nat,x\in\NN_0\}$ is dense in $\MM$.
Finally, if $\NN\subset\MM$ is an arbitrary net, by \cite{AACD2019}*{Proposition 5} we get that $\F_p(\NN)\simeq\F_p(\NN_0)$.
\end{proof}

\begin{Corollary}
Let $\MM$ be a doubling self-similar metric space and let $\NN\subset \MM$ be a net in $\MM$. For every $p\in(0,1]$, $\F_p(\MM)$ is crudely finitely representable in $\F_p(\NN)$ and $\F_p(\NN)$ is crudely finitely representable in $\F_p(\MM)$. Moreover, the finite-dimensional complemented subspace structures of $\F_p(\NN)$ and $\F_p(\MM)$ coincide; that is, there is a constant $C>1$ such that if $X$ is a finite-dimensional and $K$-complemented subspace in $\F_p(\MM)$ then there is a $(CK)$-complemented subspace $Y$ in $\F_p(\NN)$ whose Banach-Mazur distance to $X$ is at most $C$, and the other way around.
\end{Corollary}
\begin{proof}
By Theorem~\ref{thm:doubling}, we have $\F_p(\NN)\unlhd\F_p(\MM)$. The rest follows from Proposition~\ref{prop:netAndDoublingSpace}.
\end{proof}

\section{Open problems}\noindent
If $\MM$ is a compact metric space with only one accumulation point, $\F_p(\MM)$ has the commuting $C$-BAP for every $C>4^{1/p}$ (see Proposition~\ref{prop:commutingBAP}). However, by \cite{D15} and \cite{CaKa90}*{Theorem 2.4}, more can be said and $\F_1(K)$ has even the commuting MAP for every countable metric compact space $K$ (in fact, it is enough to suppose that $K$ is a countable proper metric space, see \cite{Da15}). We do not know whether a similar statement holds for $p<1$. Note that the proof for $p=1$ from \cite{D15} is based on duality techniques and so the proof for $p<1$ would be most probably interesting even for the classical case of $p=1$ as it would have to rely on different arguments.
\begin{Question}
Let $K$ be a countable proper metric space and $p\in(0,1)$. Does $\F_p(K)$ have the metric approximation property?
\end{Question}

There are known examples of metric spaces $\MM$ such that $\F(\MM)$ does not have AP. However, all the examples we know use integration techniques to some extent. For instance, integration is crucially used in the proof by Godefroy and Kalton (see \cite{GodefroyKalton2003}*{Theorem 3.1}) of the fact that $X\unlhd \F(X)$ for every separable Banach space, also in the proof by Godefroy and Ozawa (see \cite{GO14}*{Corollary 5}) that there exists a compact metric space $K$ such that $\F(K)$ fails AP, or in the construction by H\'ajek et al.\ (see \cite{HLP16}*{Corollary 2.2}) of a compact metric space homeomorphic to the Cantor space whose Lispchitz-free space fails AP. It would be interesting to find examples based on certain combinatorial features of the underlying metric space $\MM$. Since integration is not available in $p$-Banach spaces with $p<1$ (see \cite{AlbiacAnsorena2013}), a natural question in this direction is the following.
\begin{Question}
Let $p\in(0,1)$. Does there exist a metric space $\MM$ such that $\F_p(\MM)$ does not have AP?
\end{Question}

Since for uniformly discrete metric spaces $\MM$ we know that $\F_p(\MM)$ has AP (see Corollary~\ref{cor:unifDiscr}), the following seems to be an interesting problem. Note that if the answer is positive for bounded discrete metric spaces then it is positive for unbounded metric spaces with finitely many accumulation points as well (see Proposition~\ref{prop:finiteAcPoints}).
\begin{Question}
Let $\MM$ be a discrete metric space. Does $\F(\MM)$ have AP? Or, more generally, does $\F_p(\MM)$ have AP for every $p\in (0,1]$?
\end{Question}

By Theorem~\ref{thm:Banach}, for every Banach space $X$ and every $p\in(0,1]$ we have $\F_p(X)\simeq \ell_p(\F_p(X))$. Our techniques work only for metric spaces, so the following might be an interesting problem.

\begin{Question}
Let $X$ be a $p$-Banach space. Is $\F_p(X)\simeq \ell_p(\F_p(X))$ for $p\in(0,1]$?
\end{Question}

Pick a separable Banach space $X$ and $\NN_X$ a net in $X$. By \cite{AACD2019}*{Theorem 3.5} and Corollary~\ref{cor:netInCarnotGroup}, if $X$ is finite-dimensional we have $\F_p(\NN_X)\simeq \ell_p(\F_p(\NN_X))$ for every $p\in (0,1]$. The same holds for some infinite-dimensional Banach spaces $X$ and $p=1$ (see \cite{HN17}*{Theorem 8}). These results motivate us to raise the next question. Note that a similar problem has been proposed for $p=1$ in \cite{CDL19}*{Question 4} and that a positive answer for some $p<1$ would imply a positive answer for each $q\in(p,1]$.

\begin{Question}
Let $X$ be a Banach space and $\NN_X$ be a net in $X$. Is $\F_p(\NN_X)\simeq \ell_p(\F_p(\NN_X))$ for some (any) $p\in (0,1]$?
\end{Question}

By \cite{GodefroyKalton2003}*{p. 139}, we have $\F(U)\simeq U$ for Pe\l czy\'nski's universal basis space $U$. We wonder if there are more examples. Recently, there has been constructed an analogue of this space $U$ for $p$-Banach spaces, see \cite{CCM19}. However, the proof that $\F(U)\simeq U$ seems to very much depend on techniques available for Banach spaces only. Therefore, we propose an interesting question which would hopefully also shed the light onto the case $p=1$.

\begin{Question}
 Does there exist for each $p\in(0,1]$ a $p$-Banach space $X$ with $\F_p(X)\simeq X$?
\end{Question}

The following question is motivated by Corollary~\ref{cor:fddDoubling}, \cite{K15}*{Corollary 3.5}, and \cite{HP14}*{Theorem 3.1} from where it follows that whenever $K\subset\Rea^d$ is a compact set which is either countable or has a non-empty interior, then $\F(K)$ has FDD. A related question is \cite{HP14}*{Problem 4.1}, where the authors ask whether $\F(\MM)$ has a Schauder basis whenever $\MM$ is a subset of an Euclidean space.
\begin{Question}\label{question:7}
Let $d\in\Nat$. If $K\subset \Rea^d$ is an uncountable compact set with empty interior, does $\F(K)$ have the commuting BAP?
\end{Question}
Note that a positive answer to Question~\ref{question:7} would imply that for every compact set $K$ in an Euclidean space the Banach space $\F(K)$ has the FDD (see the proof of Corollary~\ref{cor:fddDoubling}).

In relation to Corollary~\ref{cor:doublingSnowflake}, we wonder if it can be extended to $p<1$. To be precise, let $0<p,\alpha\le 1$, $d\in\Nat$, and let $\MM$ be an infinite subset of $\Rea^d$.
By \cite{AACD2019}*{Theorem 3.1} we know that $\ell_p$ is complemented in $\F_p(\MM,|\cdot|^\alpha)$. We also know that $\F_p(\MM,|\cdot|^\alpha)$ is complemented in $\F_p(\Rea^d,|\cdot|^\alpha)$ by Theorem~\ref{thm:doubling}, and that $\F_p(\Rea^d,|\cdot|^\alpha)\simeq \F_p([0,1]^d,|\cdot|^\alpha)$ by Theorem~\ref{thm:Banach}. Thus, to shed light onto this question it is crucial to understand the geometry of $\F_p([0,1]^d,|\cdot|^\alpha)$. We would like to point out that the techniques used by Weaver to prove \cite{WeaverBook2018}*{Theorem 8.43} give 
\[
\F_p([0,1],|\cdot|^\alpha) \simeq\ell_p, \quad 0<p\le 1, \, 0<\alpha<1.
\]
However, unfortunately the argument breaks down for higher dimensions  if $p<1$.

\begin{Question}
Let $p$, $\alpha\in(0,1)$ and $d\in\Nat$ with $d\ge 2$. Is $\F_p([0,1]^d,|\cdot|^\alpha)$ isomorphic to $\ell_p$?
\end{Question}

As proof of Corollary~\ref{cor:doublingSnowflake} is not constructive, the following problem arises.
\begin{Question}
Let $(\MM,d)$ be a doubling metric space and $\alpha\in (0,1)$. Find $(\mu_n)_{n=1}^\infty$ in $\spn\{\delta(x)\colon x\in\MM\}\subset \F(\MM,d^\alpha)$ which is equivalent to the unit vector basis of $\ell_1$. Can we construct  $(\mu_n)_{n=1}^\infty$ so that it is also a Schauder basis  for $\F(\MM,d)$?
\end{Question}



\end{document}